\documentclass[a4paper,11pt]{article}

\usepackage[T1]{fontenc}
\usepackage[utf8]{inputenc}
\usepackage[english]{babel}
\usepackage{mleftright}
\usepackage[inline,shortlabels]{enumitem}
\usepackage{graphicx}
\usepackage{microtype}
\usepackage{amsmath,amssymb,amsthm}
\usepackage{booktabs}
\usepackage[noadjust]{cite}
\usepackage[a4paper,total={6in,9in}]{geometry}
\usepackage[hidelinks]{hyperref}

\newtheoremstyle{bolddef}{}{}{\normalfont}{}{\bfseries}{.}{ }{\thmname{#1}\thmnumber{ #2}\thmnote{ (#3)}}

\newtheoremstyle{boldplain}{}{}{\itshape}{}{\bfseries}{.}{ }{\thmname{#1}\thmnumber{ #2}\thmnote{ (#3)}}

\theoremstyle{bolddef}
\newtheorem{definition}{Definition}[section]
\newtheorem{algorithm}[definition]{Algorithm}
\newtheorem{assumption}[definition]{Assumption}

\newtheorem{remark}[definition]{Remark}

\theoremstyle{boldplain}
\newtheorem{lemma}[definition]{Lemma}
\newtheorem{theorem}[definition]{Theorem}

\newcommand{\R}{\mathbb{R}}
\newcommand{\N}{\mathbb{N}}
\newcommand{\Trp}{\mathsf{T}}
\newcommand{\minimize}{\operatornamewithlimits{minimize}}
\newcommand{\mvec}[1]{\mathbf{#1}}
\newcommand{\mmat}[1]{\mathbf{#1}}
\newcommand{\RegParam}{\mu}
\newcommand{\InitialB}{\mmat{B}_{0,k}}
\newcommand{\InitialBhat}{\hat{\mmat{B}}_{0,k}}

\begin{document}

\title{\bfseries\scshape Regularization of Limited Memory Quasi-Newton Methods for Large-Scale Nonconvex Minimization}

\date{June 5, 2022}

\author{
    Christian Kanzow$^*$\and Daniel Steck%
    \thanks{University of W\"urzburg, Institute of Mathematics, Campus Hubland Nord, Emil-Fischer-Str.\ 30, 97074 Würzburg, Germany;
    \href{mailto:kanzow@mathematik.uni-wuerzburg.de}{\nolinkurl{kanzow@mathematik.uni-wuerzburg.de}}; \href{mailto:mail@danielsteck.com}{\nolinkurl{mail@danielsteck.com}}}
}

\maketitle

{
\small\textbf{\abstractname.}
This paper deals with regularized Newton methods, a flexible class of unconstrained optimization algorithms that is competitive with line search and trust region methods and potentially combines attractive elements of both. The particular focus is on combining regularization with limited memory quasi-Newton methods by exploiting the special structure of limited memory algorithms. Global convergence of regularization methods is shown under mild assumptions and the details of regularized limited memory quasi-Newton updates are discussed including their compact representations.
Numerical results using all large-scale test problems from the CUTEst collection indicate that our regularized version of L-BFGS is competitive with state-of-the-art line search and trust-region L-BFGS algorithms and previous attempts at combining L-BFGS with regularization, while potentially outperforming some of them, especially when nonmonotonicity is involved.
\par\addvspace{\baselineskip}
}

{
\par\addvspace{\baselineskip}
}

{
\par\addvspace{\baselineskip}
}

\section{Introduction}

Let $f:\R^n\to\R$, $n\in\N$, be a twice continuously differentiable function, and consider the nonlinear minimization problem
\begin{equation}\label{Eq:Opt}
    \minimize_{\mvec{x}\in\R^n}\, f(\mvec{x}).
\end{equation}
Methods of Newton or quasi-Newton type are commonly acknowledged to be some of the most efficient algorithms for the solution of such problems.
Given a current iterate $\mvec{x}_k$, these methods compute the iteration step $\mvec{d}_k$ by solving a \emph{(quasi-)Newton equation} of the form
\begin{equation}\label{Eq:QuasiNewtonEq}
    \mmat{B}_k \mvec{d}_k = -\nabla f(\mvec{x}_k),
\end{equation}
where $\mmat{B}_k\in\R^{n\times n}$ is either the Hessian $\nabla^2 f(\mvec{x}_k)$ or an approximation thereof.
When $n$ is large, the matrix $\mmat{B}_k$ is usually not stored explicitly.
Instead, one uses so-called \emph{limited memory} quasi-Newton methods, which require the storage of a few vector pairs
\begin{equation*}
    \mvec{s}_k :=\mvec{x}_{k+1}-\mvec{x}_k, \qquad
    \mvec{y}_k :=\nabla f(\mvec{x}_{k+1})-\nabla f(\mvec{x}_k),
\end{equation*}
and use this information to construct an implicit approximation to the Hessian matrix.
This approximation is never formed explicitly; instead, the pairs $(\mvec{s}_k,\mvec{y}_k)$ are used to directly evaluate matrix--vector products of the form $\mmat{B}_k \mvec{x}$ or $\mmat{B}_k^{-1} \mvec{y}$ as necessary.
Arguably the most successful quasi-Newton schemes are the Broyden--Fletcher--Goldfarb--Shanno (BFGS) method \cite{Dennis1977} and its limited memory counterpart L-BFGS \cite{Nocedal1980,Liu1989,Byrd1994}.
Other examples include symmetric rank-one (SR1), Powell-symmetric-Broyden (PSB), Davidon--Fletcher--Powell (DFP), the so called Broyden class, and many more; see \cite{Dennis1977,Wei2006,Liu2007}.

In today's optimization landscape, L-BFGS is the de facto standard for smooth large-scale optimization.
The method is usually combined with a line search technique to ensure global convergence \cite{Liu1989}.
There have also been efforts dedicated to making quasi-Newton methods compatible with the trust-region framework; see \cite{Burke2008,Burdakov2017,Erway2014} for L-BFGS and \cite{Brust2017} for L-SR1.
This is facilitated by the fact that most quasi-Newton formulas admit a so-called \emph{compact representation} of the form
\begin{equation}\label{Eq:CompactRepresentationIntro}
    \mmat{B}_k=\InitialB+\mmat{A}_k \mmat{Q}_k^{-1}\mmat{A}_k^{\Trp},
\end{equation}
where $\InitialB\in\R^{n\times n}$, $\mmat{A}_k\in\R^{n\times s},\mmat{Q}_k\in\R^{s\times s}$ and $s\ll n$.
(We put $\mmat{Q}_k^{-1}$ instead of $\mmat{Q}_k$ in the above equation because this will be more convenient later on.)
The initial matrix $\InitialB$ is usually a multiple of the identity or some other diagonal matrix.
Decompositions of the above form have been given by many authors \cite{Byrd1994,Burdakov2002,DeGuchy2018}, and they are immensely useful in optimization methods since they usually allow the computation of matrix operations involving $\mmat{B}_k$ in the lower dimension $s$.
In particular, they facilitate the efficient computation of quasi-Newton directions and the solution of trust-region subproblems; see the references above.

In this paper, we will pursue a different globalization technique which can be seen as a (less well-known) sibling of line search and trust-region methods, the so-called \emph{regularized Newton methods} \cite{Ueda2010,Ueda2014,Li2004,Zhang2018,Li2001a}.
These are generally characterized by regularized quasi-Newton equations of the form
\begin{equation*}
    (\mmat{B}_k+\RegParam_k \mmat{I}) \mvec{d}_k = -\nabla f(\mvec{x}_k),
\end{equation*}
where $\RegParam_k\ge 0$ is called the regularization parameter.
The attractive feature of these methods is that they combine some of the respective benefits of line search and trust region methods, and moreover they are highly compatible with compact representations of quasi-Newton matrices.
We will therefore present an algorithmic framework designed to efficiently combine limited memory and regularization techniques, with the following benefits:
\begin{itemize}
\item The step computation is almost as cheap as for line search L-BFGS algorithms.
More specifically, the cost of each successful iteration (in the $m$-step BFGS case) is $4 m n$ plus the solution of a $2m \times 2m$ symmetric linear system.
In particular, no inner loop is necessary for the computation of eigenvalue decompositions or trust-region solutions.
\item At the same time, the step quality is close to that of trust-region type limited memory algorithms because the regularization parameter $\mu_k$ mimics the Lagrange multiplier arising in trust-region subproblems.
The method can therefore be considered as a kind of ``implicit'' trust-region algorithm.
\item As a result of the above, the proportion of accepted steps is very high, leading to a relatively low number of function and gradient evaluations (on a level with trust-region type methods) while at the same time preserving the ``cheap'' steps of line search methods.
\end{itemize}
The use of regularization techniques has another important benefit over line search methods.
In the line search setting, many authors advocate trying the ``full'' step size $t_k=1$ first, the motivation being that L-BFGS and similar methods are fundamentally algorithms of Newton type and the full step size may lead to fast convergence. However, the step size also serves the purpose of adapting the algorithm to the nonlinearity of the problem, and re-initializing the line search procedure with $t_k=1$ at each step makes it hard to carry this information over from one step to the next. In contrast, the regularization approach that we advocate here provides a more seamless transition between the full (quasi-)Newton step and a truncated version thereof (similar to trust region methods), which suggests that algorithms of this type may be able to handle nonlinear or nonconvex problems more effectively.

The idea of combining limited memory and regularization techniques is not entirely new. Multiple authors \cite{Li2001a,Sugimoto2014,Schraudolph2007} have advocated modifying the secant equation in quasi-Newton methods to instead approximate the sum $\nabla^2 f(\mvec{x}_k)+\RegParam_k \mmat{I}$. However, none of these methods fully exploit the quasi-Newton approximation of the Hessian and the compact representation \eqref{Eq:CompactRepresentationIntro}. The method we present takes full advantage of these tools.

In addition to the algorithm, the paper also contains a general convergence result for regularized Newton methods which, to the authors' knowledge, does not exist in this generality in the literature. In particular, the convergence result does not assume any specific quasi-Newton formula and allows for $\mmat{B}_k + \mu_k \mmat{I}$ to be indefinite. This may be of interest to researchers in the field and provide a basis for future research on related methods.

This paper is organized as follows.
Section~\ref{Sec:Method} contains a detailed description of a general class of regularized quasi-Newton methods.
Global convergence results for this class of methods are presented in Section~\ref{Sec:Conv} under fairly mild assumptions.
In Section~\ref{Sec:Matrices}, we show how compact representations of limited memory quasi-Newton methods can be exploited to create efficient implementations of the algorithm.
We also give a compact representation of the PSB formula that appears to be new.
The numerical experiments in Section~\ref{Sec:Numerics} indicate that the new technique is competitive with other attempts at regularizing L-BFGS \cite{Sugimoto2014} as well as line search and trust region based L-BFGS methods \cite{Liu1989,Burdakov2017}.
We close with some final remarks in Section~\ref{Sec:Final}.

\subsubsection*{Notation}

Matrices and vectors will be denoted by boldface letters $\mmat{M}$ and $\mvec{v}$, respectively.
Given a matrix $\mmat{M}\in\R^{s\times s}$, we write $\mmat{L}(\mmat{M})$, $\mmat{D}(\mmat{M})$,
and $\mmat{U}(\mmat{M})$ for the \emph{strictly lower, diagonal,} and \emph{strictly upper} parts of $\mmat{M}$,
respectively. In particular, it always holds that
\begin{equation*}
    \mmat{M}= \mmat{L}(\mmat{M}) + \mmat{D}(\mmat{M}) + \mmat{U}(\mmat{M}).
\end{equation*}
The gradient of the smooth function $ f $ evaluated at an iterate $ \mvec{x}_k $ will often be denoted by $ \mvec{g}_k $. We denote sequences by $\{s_k\}$ and write $\{s_k\}_{k\in\mathcal{S}}$ for the subsequence induced by an infinite index set $\mathcal{S} = \{k_1, k_2, \ldots\}\subseteq\N$ with $k_i < k_{i+1}$ for all $i$. Similarly, $s_k \to_{\mathcal{S}} s$ means that $\{s_k\}_{k\in\mathcal{S}}$ converges to $s$.

\section{Regularized Quasi-Newton Methods}\label{Sec:Method}

As discussed in the introduction, the fundamental principle underlying the methods in this paper is that of regularized Newton and quasi-Newton methods, which are generally characterized by regularized quasi-Newton equations of the form
\begin{equation}\label{Eq:BasicRegularizedNewton}
    (\mmat{B}_k+\RegParam_k \mmat{I}) \mvec{d}_k = -\nabla f(\mvec{x}_k),
\end{equation}
where $\mmat{B}_k$ is either the Hessian $\nabla^2 f(\mvec{x}_k)$ or an approximation thereof, and $\RegParam_k\ge 0$ is the \emph{regularization parameter}. Clearly, if $\RegParam_k=0$, then \eqref{Eq:BasicRegularizedNewton} reduces to the standard quasi-Newton equation $\mmat{B}_k \mvec{d}_k = -\nabla f(\mvec{x}_k)$. On the other hand, if $\RegParam_k\gg 0$ is large, then the matrix $\mmat{B}_k+\RegParam_k \mmat{I}$ will be invertible, and the step $\mvec{d}_k$ produced by \eqref{Eq:BasicRegularizedNewton} will essentially be the negative gradient direction (up to normalization; see Lemma~\ref{Lem:AsymptoticGradient}).

\subsection{Mathematical Motivation}

The virtues of the regularization approach can be understood by recognizing that this essentially amounts to minimizing the regularized quadratic model
\begin{equation}\label{Eq:RegularizedQuadratic}
    \hat{q}_k(\mvec{d}):= f(\mvec{x}_k)+\mvec{g}_k^{\Trp}\mvec{d}+ \frac{1}{2}\mvec{d}^{\Trp}\mmat{B}_k\mvec{d}+\frac{\RegParam_k}{2}\|\mvec{d}\|^2,
\end{equation}
which differs from the conventional Newton model by Tikhonov regularization. Thus, a positive value of $\RegParam_k$ may dampen the impact of negative eigenvalues of $\mmat{B}_k$ on the search direction, prevent excessively long steps in negative curvature directions, and possibly guarantee that the model \eqref{Eq:RegularizedQuadratic} admits a unique minimizer (i.e., that the matrix $\mmat{B}_k+\RegParam_k\mmat{I}$ is positive definite). The anticipated setting is that $\RegParam_k$ will initially be kept sufficiently large to guarantee global convergence, eventually decreasing rapidly enough so as to not impede fast local convergence.

A more rigorous interpretation is given by trust-region methods. Indeed, if $\mvec{d}_k:=-(\mmat{B}_k+\RegParam_k\mmat{I})^{-1} \mvec{g}_k$ for some $\RegParam_k\ge 0$, and if $\Delta:=\|\mvec{d}_k\|$, then $\mvec{d}_k$ is a stationary point of the \emph{trust-region subproblem}
\begin{equation*}
    \minimize_{\|\mvec{d}\|\le \Delta}\, q_k(\mvec{d}),
\end{equation*}
where
\begin{equation}\label{Eq:Quadratic}
    q_k(\mvec{d}):=f(\mvec{x}_k)+\mvec{g}_k^{\Trp}\mvec{d}+\frac{1}{2}\mvec{d}^{\Trp}\mmat{B}_k\mvec{d}
\end{equation}
is the standard quadratic approximation of $f$ around $\mvec{x}_k$.
If $\mmat{B}_k+\RegParam_k \mmat{I}$ is positive definite, then $\mvec{d}_k$ is in fact a \emph{solution} of this auxiliary problem. It follows that regularized Newton methods can be interpreted as ``implicit'' trust-region methods whereby the regularization parameter is controlled instead of the trust-region radius.

Finally, it is also interesting to analyze how the regularization technique affects the conditioning of the quadratic model \eqref{Eq:RegularizedQuadratic}. Assuming for the moment that $\mmat{B}_k$ is positive definite (as it is, e.g., in BFGS-type methods), the regularization parameter improves the condition number of the underlying matrix in the sense that
\begin{equation*}
    \kappa(\mmat{B}_k+\RegParam_k \mmat{I})=\frac{\lambda_{\max}(\mmat{B}_k)+\RegParam_k}{\lambda_{\min}(\mmat{B}_k)+\RegParam_k} \le
    \frac{\lambda_{\max}(\mmat{B}_k)}{\lambda_{\min}(\mmat{B}_k)}=\kappa(\mmat{B}_k),
\end{equation*}
where $\lambda_{\max}(\mmat{B}_k),\lambda_{\min}(\mmat{B}_k)>0$ are the largest and smallest eigenvalues of $\mmat{B}_k$, respectively.

\subsection{Basic Algorithm}

To control the regularization parameter $\RegParam_k$, we consider the quadratic approximation $q_k$ of $f$ from \eqref{Eq:Quadratic} and borrow some terminology from trust-region algorithms.
Given a candidate step $\mvec{d}_k=-(\mmat{B}_k+\RegParam_k\mmat{I})^{-1}\mvec{g}_k$, define the \emph{predicted reduction} of $f$ as
\begin{equation}\label{Eq:PRed}
    \textnormal{pred}_k:=f(\mvec{x}_k)-q_k(\mvec{d}_k)
    = -\mvec{g}_k^{\Trp}\mvec{d}_k - \frac{1}{2}\mvec{d}_k^{\Trp}\mmat{B}_k \mvec{d}_k
    = \frac{\RegParam_k}{2}\|\mvec{d}_k\|^2-\frac{1}{2}\mvec{g}_k^{\Trp}\mvec{d}_k,
\end{equation}
where the last equality uses the definition of $\mvec{d}_k$. (Note that, in particular, the matrix $\mmat{B}_k$ need not be available for the computation of $\textnormal{pred}_k$.) This quantity will be compared to the \emph{actual} or \emph{achieved reduction} in step $k$,
\begin{equation}\label{Eq:ARed}
    \textnormal{ared}_k := f(\mvec{x}_k)-f(\mvec{x}_k+\mvec{d}_k).
\end{equation}
Similar to trust-region methods \cite{Conn2000}, we use the ratio between these quantities to control the regularization parameter. To this end, we distinguish between three cases, unsuccessful \textbf{(u)}, successful \textbf{(s)}, and highly successful \textbf{(h)} steps. Special care also needs to be taken because there is no a-priori guarantee that $\textnormal{pred}_k$ is positive (since $\mmat{B}_k$ may be indefinite); such steps are treated in the same manner as unsuccessful ones.

\begin{algorithm}[Regularized quasi-Newton method]\label{Alg:RegularizedQN}\leavevmode\\
    Choose $\mvec{x}_0\in\R^n$ and parameters $\RegParam_0>0$; $p_{\min},c_1\in (0,1)$; $c_2\in (c_1,1)$; $\sigma_1\in (0,1)$; $\sigma_2>1$.
\begin{enumerate}[topsep=1ex,parsep=0ex,leftmargin=*,label=\textbf{Step~\arabic*.}]
    \item If a suitable stopping criterion is satisfied, terminate.
    \item \textbf{(Step computation)} Choose $\mmat{B}_k\in\R^{n\times n}$ and attempt to solve the regularized
    quasi-Newton equation
\begin{equation}\label{Eq:RegularizedNewtonEq}
    (\mmat{B}_k+\RegParam_k \mmat{I}) \mvec{d}_k=-\nabla f(\mvec{x}_k).
\end{equation}
    If this equation admits no solution $\mvec{d}_k$, or if $\textnormal{pred}_k \le p_{\min} \|\mvec{g}_k\| \|\mvec{d}_k\|$, set $\mvec{x}_{k+1}:=\mvec{x}_k$, $\RegParam_{k+1}:=\sigma_2\RegParam_k$, and go to \emph{Step~4}. Otherwise, go to \emph{Step~3}.

\item \textbf{(Variable update)} Set $\varrho_k:=\textnormal{ared}_k/\textnormal{pred}_k$ and perform one of the following steps:

\textbf{Step~3u ($\varrho_k\le c_1$).} Set $\mvec{x}_{k+1}:=\mvec{x}_k$ and $\RegParam_{k+1}:=\sigma_2 \RegParam_k$.

\textbf{Step~3s ($c_1<\varrho_k\le c_2$).} Set $\mvec{x}_{k+1}:=\mvec{x}_k+\mvec{d}_k$ and $\RegParam_{k+1}:=\RegParam_k$.

\textbf{Step~3h ($c_2<\varrho_k$).} Set $\mvec{x}_{k+1}:=\mvec{x}_k+\mvec{d}_k$ and $\RegParam_{k+1}:=\sigma_1\RegParam_k$.

\item Set $k\leftarrow k+1$ and go to Step~1.
\end{enumerate}
\end{algorithm}

The condition $\textnormal{pred}_k > p_{\min} \|\mvec{g}_k\| \|\mvec{d}_k\|$ in Step~2 is a sufficient descent criterion similar to the angle condition in line search methods or the Cauchy condition in trust-region methods.
The quantity $\textnormal{pred}_k$ denotes the minimal expected reduction in objective value (relative to $\mvec{g}_k$ and $\mvec{d}_k$) for a step to be attempted.

As hinted above, in what follows, we will refer to a step as \emph{unsuccessful} if it passes through Step~3u or skips Step~3 because of the checks in Step~2. (In particular, $\mvec{d_k}$ may not be defined in an unsuccessful step.)

The parameters $c_1,c_2,\sigma_1, \sigma_2$ are used to classify steps and adjust the regularization accordingly (increase if the step was unsuccessful, decrease if the step was highly successful).

Algorithm~\ref{Alg:RegularizedQN} is closely related to trust-region methods.
The main difference between trust-region methods and our regularization framework lies in the update of the parameter $ \RegParam_k $.
The former uses an indirect way to compute $ \RegParam_k $ (via a trust-region radius), whereas here we update the regularization parameter directly.
While the indirect update follows a well-understood and well-motivated philosophy, its actual computation is sometimes time-consuming and costly.
We therefore expect superior behavior of the direct update, in particular for large-scale problems.

The report \cite{Sugimoto2014} presents a method which is formally almost identical (except for a slightly different update of the regularization parameter) to Algorithm~\ref{Alg:RegularizedQN}.
The main difference is that \cite{Sugimoto2014} focuses on the matrices $ \mmat{B}_k $ being updated by a limited memory BFGS scheme (without using compact representations, as we shall do in Section~\ref{Sec:Matrices}).
The convergence theory in \cite{Sugimoto2014} assumes a bounded level set condition; this is not required in our subsequent analysis, which is substantially more general since we only assume boundedness of $\{\mmat{B}_k\}$ (allowing for other quasi-Newton formulas or indefiniteness) and boundedness of the objective from below (consider, for example, the exponential function).


\section{General Convergence Analysis}\label{Sec:Conv}

As we shall see, Algorithm~\ref{Alg:RegularizedQN} provides a powerful framework for the application of
quasi-Newton type updates. Before turning to this discussion (which is the main motivation for this paper),
we shall dedicate the present section to a simple convergence analysis.
Due to the non-specificity of the algorithm in its general form, it will be convenient to carry out the convergence analysis under rather general assumptions. To this end, we shall make no assumption on the particular choice of the matrices $\mmat{B}_k$, which may or may not be approximations of the Hessian $\nabla^2 f(\mvec{x}_k)$. The only assumption we make throughout this section is the following.

\begin{assumption}[Boundedness]\label{Asm:Conv}
    $\{\mmat{B}_k\}\subseteq\R^{n\times n}$ is a bounded sequence.
\end{assumption}

Most practically relevant quasi-Newton schemes should have no issues satisfying the above assumption, especially when the gradient $\nabla f$ is Lipschitz continuous on an appropriate level set. Indeed, many of these techniques yield Hessian approximations which satisfy additional properties such as symmetry (which we omitted because it is unnecessary for the theory below) or positive definiteness.

\begin{lemma}[Gradient approximation]\label{Lem:AsymptoticGradient}
    Let Assumption~\ref{Asm:Conv} hold, and let $\RegParam_k\to \infty$. Then $\mmat{B}_k+\RegParam_k\mmat{I}$ is invertible for sufficiently large $k\in\N$, and
\begin{equation*}
    \lim_{k\to\infty} \frac{(\mmat{B}_k+\RegParam_k\mmat{I})^{-1}\mvec{z}}{\|(\mmat{B}_k+\RegParam_k\mmat{I})^{-1}\mvec{z}\|}=\frac{\mvec{z}}{\|\mvec{z}\|}
    \quad\text{for all }\mvec{z}\in\R^n\setminus\{0\}.
\end{equation*}
\end{lemma}

The above result defines more precisely the intuitive relationship mentioned in Section~\ref{Sec:Method}; that is, if the regularization parameter is sufficiently large, then the regularized Newton equation \eqref{Eq:RegularizedNewtonEq} admits a unique solution, and the resulting vector will approximate the negative gradient direction as $\RegParam_k\to\infty$.

Another consequence of Lemma~\ref{Lem:AsymptoticGradient} is that the method performs infinitely many
successful steps. This follows from the fact that $\mvec{d}_k$ becomes ever smaller and approaches the
(local) steepest descent direction when $\RegParam_k\to\infty$, thus leading to a local descent step
which satisfies the sufficient decrease condition from Step~2 of the algorithm.

\begin{lemma}[Well-definedness]\label{Lem:WellDefinedness}
    Let Assumption~\ref{Asm:Conv} hold, and assume that $\mvec{g}_k\ne 0$ for all $k$. Then Algorithm~\ref{Alg:RegularizedQN} performs infinitely many successful or highly successful steps.
\end{lemma}
\begin{proof}
    Assume for the sake of contradiction that there exists $k_0\in\N$ such that all steps $k\ge k_0$ are unsuccessful. In particular, this implies $\RegParam_k\to \infty$ as $k\to\infty$ and $\mvec{x}_k=\mvec{x}_{k_0}$ for all $k\ge k_0$. Since $\{\mmat{B}_k\}$ is a bounded sequence, it follows from Lemma~\ref{Lem:AsymptoticGradient} that $\mmat{B}_k+\RegParam_k \mmat{I}$ is invertible for sufficiently large $k$, that $\mvec{d}_k\to 0$, and $\mvec{d}_k/\|\mvec{d}_k\| \to - \mvec{g}_{k_0}/\|\mvec{g}_{k_0}\|$. Moreover, the regularized Newton equation \eqref{Eq:RegularizedNewtonEq} implies that $\RegParam_k \|\mvec{d}_k\|\to \|\mvec{g}_{k_0}\|$. It is easy to deduce from these limit relations that
\begin{equation*}
    \textnormal{pred}_k=\frac{\RegParam_k}{2}\|\mvec{d}_k\|^2-\frac{1}{2}\mvec{g}_k^{\Trp}\mvec{d}_k > p_{\min} \|\mvec{g}_k\| \|\mvec{d}_k\|
    \quad\text{for sufficiently large }k
\end{equation*}
    (simply divide this inequality by $ \| \mvec{d}_k \| $ and recall that $ p_{\min} \in (0,1) $).
    Hence, the algorithm must eventually perform only Step~3u, which means that $\textnormal{ared}_k\le c_1 \textnormal{pred}_k$ for all $k\ge k_0$ sufficiently large. It then follows that
\begin{equation}\label{Eq:LemWellDefinedness1}
    f(\mvec{x}_{k_0}+\mvec{d}_k)-f(\mvec{x}_{k_0}) = -\textnormal{ared}_k \ge - c_1 \textnormal{pred}_k = \frac{c_1}{2}\mvec{g}_{k_0}^{\Trp}\mvec{d}_k-\frac{c_1 \RegParam_k}{2}\|\mvec{d}_k\|^2
    \quad\text{for }k\ge k_0.
\end{equation}
    We now divide both sides of this inequality by $t_k:=\|\mvec{d}_k\|$. Recalling that $\mvec{d}_k/\|\mvec{d}_k\| \to -\mvec{g}_{k_0}/\|\mvec{g}_{k_0}\|$, it follows that the left-hand side becomes
\begin{equation}\label{Eq:LemWellDefinedness2}
    \frac{f\mleft( \mvec{x}_{k_0}+t_k \frac{\mvec{d}_k}{\|\mvec{d}_k\|} \mright)-f(\mvec{x}_{k_0})}{t_k} \to \nabla f(\mvec{x}_{k_0})^{\Trp}\frac{-\mvec{g}_{k_0}}{\|\mvec{g}_{k_0}\|}=-\|\mvec{g}_{k_0}\|.
\end{equation}
    Conversely, recalling that $\RegParam_k \|\mvec{d}_k\|\to \|\mvec{g}_{k_0}\|$, the right-hand side of \eqref{Eq:LemWellDefinedness1} divided by $t_k$ satisfies
\begin{equation}\label{Eq:LemWellDefinedness3}
    \frac{c_1}{2}\mvec{g}_{k_0}^{\Trp}\frac{\mvec{d}_k}{\|\mvec{d}_k\|}-\frac{c_1 \RegParam_k}{2}\|\mvec{d}_k\| \to \frac{c_1}{2} \mvec{g}_{k_0}^{\Trp}\frac{-\mvec{g}_{k_0}}{\|\mvec{g}_{k_0}\|}-\frac{c_1}{2} \|\mvec{g}_{k_0}\|=-c_1 \|\mvec{g}_{k_0}\|.
\end{equation}
    Since $c_1\in (0,1)$, it then follows from \eqref{Eq:LemWellDefinedness2}, \eqref{Eq:LemWellDefinedness3} that $\|\mvec{g}_{k_0}\|=0$, a contradiction.
\end{proof}

The following result builds upon the well-definedness of the algorithm and shows that it achieves asymptotic stationarity.

\begin{theorem}[Global convergence I]\label{thm:convI}
    Let Assumption~\ref{Asm:Conv} hold, let $f$ be bounded from below, and $\{\mvec{x}_k\}$ generated by Algorithm~\ref{Alg:RegularizedQN}.
    Then $\liminf_{k\to\infty}\|\mvec{g}_k\|= 0$; in particular, given any $\varepsilon>0$, the algorithm terminates with $\|\mvec{g}_k\|<\varepsilon$ after finitely many iterations.
\end{theorem}
\begin{proof}
    Let $\mathcal{S}\subseteq\N$ be the set of indices of successful or highly successful steps. Note that $|\mathcal{S}|=\infty$ by Lemma~\ref{Lem:WellDefinedness}. Assume for the sake of contradiction that
\begin{equation}\label{Eq:ThmGlobalConvContradictionAsm}
    \liminf_{k\to\infty}\|\mvec{g}_k\|>0.
\end{equation}
    Since every step $k\in\mathcal{S}$ is successful, we have by definition that
\begin{equation*}
    f(\mvec{x}_k)-f(\mvec{x}_{k+1})\ge c_1 \textnormal{pred}_k \ge p_{\min} c_1 \|\mvec{g}_k\| \|\mvec{d}_k\|
    \quad\text{for every }k\in\mathcal{S}.
\end{equation*}
    By \eqref{Eq:ThmGlobalConvContradictionAsm}, there exist $k_0\in\N$ and $\varepsilon>0$ such that $\|\mvec{g}_k\|\ge \varepsilon$ for all $k\ge k_0$. Using the fact that $f$ is bounded from below, we obtain
\begin{equation}\label{Eq:ThmGlobalConvCauchy}
    \infty > \sum_{k\in\N} \big[ f(\mvec{x}_k)-f(\mvec{x}_{k+1}) \big] =
    \sum_{k\in\mathcal{S}} \big[ f(\mvec{x}_k)-f(\mvec{x}_{k+1}) \big]
    \ge p_{\min} c_1 \varepsilon \sum_{k\in\mathcal{S},\,k\ge k_0} \|\mvec{d}_k\|
\end{equation}
    and, in particular, $\mvec{d}_k\to_{\mathcal{S}}0$. Since every step $k\in\mathcal{S}$ is successful, we have $(\mmat{B}_k+\RegParam_k \mmat{I})\mvec{d}_k = -\mvec{g}_k$ for all $k\in\mathcal{S}$. This implies that $\{\RegParam_k\}_{k\in\mathcal{S}}$ cannot have a bounded subsequence (since this together with $\mvec{d}_k\to_{\mathcal{S}}0$ would violate \eqref{Eq:ThmGlobalConvContradictionAsm}). Hence, $\RegParam_k\to_{\mathcal{S}}+\infty$. In particular, the algorithm also performs infinitely many unsuccessful steps (i.e., $|\N\setminus\mathcal{S}|=\infty$), and $\RegParam_k\to +\infty$ since $\RegParam_k$ cannot decrease during unsuccessful iterations.

    Now, since $\mathcal{S}$ and $\N\setminus\mathcal{S}$ are infinite, we may choose an infinite set
    $\mathcal{S}'\subseteq\mathcal{S}$ such that $k-1\in\N\setminus\mathcal{S}$ whenever $k\in\mathcal{S}'$.
    Since $\mvec{x}_k$ is not updated in unsuccessful steps, it follows from \eqref{Eq:ThmGlobalConvCauchy} that
    \begin{equation*}
       \infty > p_{\min} c_1 \varepsilon \sum_{k\in\mathcal{S},\,k\ge k_0} \|\mvec{d}_k\| =
       p_{\min} c_1 \varepsilon \sum_{k\in\mathcal{S},\,k\ge k_0} \|\mvec{x}_{k+1} - \mvec{x}_k \| =
       p_{\min} c_1 \varepsilon \sum_{k\ge k_0} \|\mvec{x}_{k+1} - \mvec{x}_k \| .
    \end{equation*}
    Hence $\{\mvec{x}_k\}_{k\in\N}$ is a Cauchy sequence, and thus convergent.
    Let $\bar{\mvec{x}}$ denote its limit point. In particular, we then obtain
    $\mvec{x}_{k-1}\to_{\mathcal{S}'}\bar{\mvec{x}}$; thus, using $\RegParam_k\to +\infty$ and arguing as in
    the proof of Lemma~\ref{Lem:WellDefinedness},
    it follows that the steps $k-1$, $k\in\mathcal{S}'$, must be successful for sufficiently large $k\in\mathcal{S}'$. This is a contradiction.
\end{proof}

Note that the counterpart of Theorem~\ref{thm:convI} also holds for trust-region methods under the same
set of assumptions. Moreover, the technique of proof used here is related to the corresponding one known
for trust-region methods. Nevertheless, we stress that one has to be careful in translating the standard
trust-region proof to our regularization framework since well-known properties of the solution of the
trust-region subproblem may not hold in our case.

Similar to the theory of trust-region methods, we can use Theorem~\ref{thm:convI} to obtain a stronger
convergence result under an additional assumption.

\begin{theorem}[Global convergence II]
    Let Assumption~\ref{Asm:Conv} hold, let $f$ be bounded from below, and $\{\mvec{x}_k\}$ generated by Algorithm~\ref{Alg:RegularizedQN}.
    Suppose that $ \nabla f $ is uniformly continuous on a set $ X \subseteq \R^n $ satisfying $ \{\mvec{x}_k\} \subseteq X $.
    Then $\lim_{k\to\infty}\|\mvec{g}_k\|= 0$; in particular, every accumulation point of $\{\mvec{x}_k\}$ is a stationary point of $ f $.
\end{theorem}
\begin{proof}
Assume there exists $ \delta > 0 $ and a subsequence $ \{ \mvec{x}_k \}_{k\in K} $ such that
\begin{equation*}
   \| \mvec{g}_k \| \geq 2 \delta \quad \text{for all } k \in K.
\end{equation*}
Since $ \liminf_{k\to\infty}\|\mvec{g}_k\|= 0 $ by Theorem~\ref{thm:convI}, we can find, for each $ k \in K $,
an index $ \ell (k) > k $ such that
\begin{equation*}
   \| \mvec{g}_l \| \geq \delta \quad \text{for all } k \leq l < \ell (k), \qquad \text{and} \qquad
   \| \mvec{g}_{\ell (k)} \| < \delta, \quad k \in K.
\end{equation*}
For an arbitrary $ k \in K $ and a successful or highly successful iteration $ l $ with $ k \leq l < \ell (k) $,
we obtain
\begin{equation*}
   f(\mvec{x}_l) - f(\mvec{x}_{l+1}) \geq c_1 \textnormal{pred}_k \geq p_{\min} c_1 \| \mvec{g}_l \| \| \mvec{d}_l \|
   \geq p_{\min} c_1 \delta \| \mvec{x}_{l+1} - \mvec{x}_l \|.
\end{equation*}
The same inequality holds for $ l $ being unsuccessful simply because $ \mvec{x}_{l+1} = \mvec{x}_l $ in this case. This implies
\begin{align*}
   p_{\min} c_1 \delta \| \mvec{x}_{\ell(k)} - \mvec{x}_k \|
   \le p_{\min} c_1 \delta \sum_{l= k}^{\ell(k)-1} \| \mvec{x}_{l+1} - \mvec{x}_l \|
   & \le \sum_{l= k}^{\ell(k)-1} \big( f(\mvec{x}_l) - f(\mvec{x}_{l+1}) \big) \\
   & = f(\mvec{x}_k) - f(\mvec{x}_{\ell(k)})
\end{align*}
for all $ k \in K $. Since $ f $ is bounded from below and $ \{ f(\mvec{x}_k) \} $ is monotonically
decreasing, we obtain $ f(\mvec{x}_k) - f(\mvec{x}_{\ell(k)}) \to 0 $ for $ k \to \infty $. This implies
$ \| \mvec{x}_{\ell(k)} - \mvec{x}_k \| \to_K 0 $. The uniform continuity of $ \nabla f $ on the set $ X $
therefore yields
\begin{equation*}
   \| \nabla f ( \mvec{x}_{\ell(k)} ) - \nabla f( \mvec{x}_k ) \| \to_K 0.
\end{equation*}
On the other hand, the choice of the index $ \ell(k) $ implies
\begin{equation*}
   \| \nabla f ( \mvec{x}_{\ell(k)} ) - \nabla f( \mvec{x}_k ) \| \geq \| \nabla f( \mvec{x}_k ) \| -
    \| \nabla f ( \mvec{x}_{\ell(k)} ) \| \geq 2 \delta - \delta = \delta .
\end{equation*}
This contradiction completes the proof.
\end{proof}

We close this section by noting that regularization techniques like in Algorithm~\ref{Alg:RegularizedQN}
are sometimes used in order to prove local fast convergence properties for Newton-type methods. This
corresponds to the choice $ \mmat{B}_k := \nabla^2 f(\mvec{x}_k) $ as the exact Hessian. Using a more refined
update of the regularization parameter, assuming a local error bound condition and the Hessian of $ f $
to be locally Lipschitz continuous, it is possible to verify local quadratic convergence for convex
objective functions, cf.~\cite{Li2004,Ueda2010,Ueda2014}. Since our focus is on large-scale problems,
our subsequent analysis concentrates on $ \mmat{B}_k $ being computed by limited memory quasi-Newton
matrices.


\section{Regularized Quasi-Newton Matrices}\label{Sec:Matrices}

This section provides the details of \emph{limited memory} type implementations of quasi-Newton methods. Some of the material below can be applied with minimal modifications to full memory quasi-Newton methods, but we forgo these investigations due to our focus on large-scale optimization.

In keeping with conventional limited memory notation, we assume an algorithmic framework where the last $m$ variable steps $\mvec{s}_i:=\mvec{x}_{i+1}-\mvec{x}_i$ are tracked together with the corresponding gradient differences $\mvec{y}_i:=\mvec{g}_{i+1}-\mvec{g}_i$, where we recall that $\mvec{g}_i=\nabla f(\mvec{x}_i)$. For convenience of notation, we aggregate these in the matrices
\begin{equation*}
    \mmat{S}_k := [\mvec{s}_{k-m} \,\cdots\, \mvec{s}_{k-1}] \in\R^{n\times m} \quad\text{and}\quad
    \mmat{Y}_k := [\mvec{y}_{k-m} \,\cdots\, \mvec{y}_{k-1}]\in\R^{n\times m}.
\end{equation*}
If fewer than $m$ previous iterates are available, that is, if $k<m$, we set
\begin{equation*}
    \mmat{S}_k := [\mvec{s}_0 \,\cdots\, \mvec{s}_{k-1}]\in \R^{n\times k} \quad\text{and}\quad
    \mmat{Y}_k := [\mvec{y}_0 \,\cdots\, \mvec{y}_{k-1}]\in \R^{n\times k}.
\end{equation*}
These definitions may seem like a mere matter of notation, but there are actually quite pragmatic arguments why $\mmat{S}$ and $\mmat{Y}$ should be treated as matrices instead of collections of vectors. Many limited memory operations can be formulated as loops over the recurring index $i=1,\ldots,m$, and the matrix notation sometimes allows us to formulate the underlying calculations as \emph{matrix--vector} operations (instead of a sequence of vector--vector operations). This approach should be used whenever possible in practical implementations because it leverages the power of low-level BLAS (basic linear algebra subprograms) and parallelism, providing a significant increase in computational efficiency.

\begin{remark}[Rejected quasi-Newton updates]\label{Rem:RejectedQNUpdates}
    For the sake of simplicity and to avoid notational overhead, we assume that the algorithm always ``accepts'' the data pair $(\mvec{s}_k,\mvec{y}_k)$ in each successful iteration. This is not the case for some quasi-Newton schemes, especially for nonconvex objective functions. In general, quasi-Newton updates are typically accepted or rejected using a so-called cautious updating scheme (see Section~\ref{Sec:Numerics}); when a pair $(\mvec{s}_k, \mvec{y}_k)$ is rejected, the matrices $\mmat{S}_k,\mmat{Y}_k$ of previous steps simply remain as they were.
\end{remark}

Most limited memory quasi-Newton methods implicitly generate a so-called \emph{compact representation} of the form
\begin{equation}\label{Eq:CompactRepresentation}
    \mmat{B}_k=\InitialB+\mmat{A}_k \mmat{Q}_{k}^{-1} \mmat{A}_k^{\Trp},
\end{equation}
where $\mmat{Q}_k\in\R^{s\times s}$ is a nonsingular symmetric matrix, $\mmat{A}_k\in\R^{n\times s}$, and $s\ll n$ is a constant depending on the particular quasi-Newton scheme. For instance, $s=2 m$ in limited memory BFGS methods, and $s=m$ for limited memory SR1.

The above representation provides a very convenient framework for the regularization approach: given a parameter $\RegParam\ge 0$ (e.g., one of the values $\RegParam_k$ from Algorithm~\ref{Alg:RegularizedQN}), the regularized Hessian approximation can be written as
\begin{equation*}
    \mmat{B}_k+\RegParam \mmat{I} =(\InitialB+\RegParam \mmat{I})
    +\mmat{A}_k \mmat{Q}_k^{-1} \mmat{A}_k^{\Trp}.
\end{equation*}
This facilitates the application of low-rank update formulas to compute the regularized Newton step both explicitly and cheaply.
To this end, let $\hat{\mmat{B}}_k:=\mmat{B}_k+\RegParam\mmat{I}$ and $\InitialBhat:=\InitialB+\RegParam\mmat{I}$.
Then the Sherman--Morrison--Woodbury formula implies that
\begin{equation}\label{Eq:RegularizedQNinverse}
    \hat{\mmat{B}}_k^{-1}=\InitialBhat^{-1}-
    \InitialBhat^{-1} \mmat{A}_k (\mmat{Q}_k
    +\mmat{A}_k^{\Trp}\InitialBhat^{-1} \mmat{A}_k)^{-1} \mmat{A}_k^{\Trp} \InitialBhat^{-1}\,,
\end{equation}
provided that $\InitialBhat$ is nonsingular.
Note that $\InitialBhat$ is usually a diagonal matrix whose inversion is trivial. Moreover, the inner matrix $\mmat{Q}_k+\mmat{A}_k^{\Trp} \InitialBhat^{-1}\mmat{A}_k$ is of size $s\times s$, so that its inversion can be carried out cheaply in relation to the dimension $n$. By the Woodbury matrix identity, the invertibility of this inner matrix is equivalent to that of $\hat{\mmat{B}}_k$.

In the following, we shall mainly assume that the initial matrix $\InitialB$ is chosen as a scalar multiple of the identity, $\InitialB:=\gamma_k \mmat{I}$. Writing $\hat{\gamma}_k:=\gamma_k+\RegParam$, it then follows that
\begin{equation}\label{Eq:RegularizedQNinverse_I}
    \hat{\mmat{B}}_k^{-1}=\hat{\gamma}_k^{-1}\mmat{I} - \hat{\gamma}_k^{-2}
    \mmat{A}_k (\mmat{Q}_k+ \hat{\gamma}_k^{-1} \mmat{A}_k^{\Trp} \mmat{A}_k)^{-1} \mmat{A}_k^{\Trp}.
\end{equation}
The practical efficiency of quasi-Newton methods significantly depends on the memorization and re-use of previously computed quantities. To this end, observe that the quasi-Newton recurrence implies
\begin{equation}\label{Eq:RegularizedQNstep}
    \mvec{s}_k = -\hat{\mmat{B}}_k^{-1} \mvec{g}_k= -\hat{\gamma}_k^{-1}\mvec{g}_k+
    \hat{\gamma}_k^{-2} \mmat{A}_k \mvec{p}_k,
\end{equation}
where
\begin{equation}\label{Eq:RegularizedQNstep_p}
    \mvec{p}_k :=(\mmat{Q}_k+\hat{\gamma}_k^{-1}\mmat{A}_k^{\Trp}\mmat{A}_k)^{-1}
    \mmat{A}_k^{\Trp} \mvec{g}_k.
\end{equation}
Thus, the main computational cost occurs in forming the product $\mmat{A}_k^{\Trp}\mvec{g}_k$, the solution of an $s\times s$ symmetric linear equation to obtain $\mvec{p}_k$, and the product $\mmat{A}_k \mvec{p}_k$. In addition, the matrices $\mmat{A}_k$ and $\mmat{Q}_k$ need to be updated in each iteration, and the matrix $\mmat{A}_k^{\Trp}\mmat{A}_k$ needs to be available. As we shall see later, it is possible to reduce the cost of these computations by using the inherent dependencies between the underlying formulas.

\begin{remark}[Regularized secant equation]\label{Rem:ModifiedSecant}
    Instead of compact representations, it is also possible to combine regularization and quasi-Newton techniques by directly approximating the sum $\nabla^2 f(\mvec{x}_k)+\RegParam \mmat{I}$; see \cite{Sugimoto2014}. This idea is based on the fact that the regularized Hessian satisfies (approximately) the modified secant equation
\begin{equation*}
    (\nabla^2 f(\mvec{x}_k)+\RegParam \mmat{I}) \mvec{s}_k
    \approx \mvec{y}_k +\RegParam \mvec{s}_k.
\end{equation*}
    Thus, an approximation $\hat{\mmat{B}}_k$ to $\nabla^2 f(\mvec{x}_k)+\RegParam \mmat{I}$ can be constructed by taking a modified initial guess $\InitialBhat:=\InitialB+\RegParam \mmat{I}$ and applying an arbitrary quasi-Newton scheme to the modified pair $(\mmat{S}_k,\hat{\mmat{Y}}_k):=(\mmat{S}_k,\mmat{Y}_k+\RegParam \mmat{S}_k)$. For certain quasi-Newton schemes like SR1 and PSB, this actually yields the same results as the approach based on compact representations (see Sections~\ref{Sec:Matrices:SR1} and \ref{Sec:Matrices:PSB}). In general, however, the two approaches are different.
\end{remark}

\subsection{Broyden--Fletcher--Goldfarb--Shanno (BFGS)}\label{Sec:Matrices:BFGS}

The BFGS update is often considered the most successful quasi-Newton scheme.
Throughout this section, let $\InitialB = \gamma_k \mmat{I}$ for some $\gamma_k\in\R$.
Following \cite{Byrd1994}, the compact representation of L-BFGS is given by
\begin{equation}\label{Eq:CompactRepresentationBFGS}
    \mmat{B}_k=\gamma_k \mmat{I}-
\begin{bmatrix}
    \mmat{S}_k & \mmat{Y}_k
\end{bmatrix}
\begin{bmatrix}
    \gamma_k^{-1} \mmat{S}_k^{\Trp} \mmat{S}_k & \gamma_k^{-1}\mmat{L}_k \\
    \gamma_k^{-1}\mmat{L}_k^{\Trp} & -\mmat{D}_k
\end{bmatrix}^{-1}
\begin{bmatrix}
    \mmat{S}_k^{\Trp} \\[1pt] \mmat{Y}_k^{\Trp}
\end{bmatrix},
\end{equation}
    where
\begin{equation}\label{Eq:AuxiliaryMatricesQN}
    \mmat{D}_k:= \mmat{D}(\mmat{S}_k^{\Trp}\mmat{Y}_k)
    \quad\text{and}\quad \mmat{L}_k:=\mmat{L}(\mmat{S}_k^{\Trp}\mmat{Y}_k)
\end{equation}
(recall that $ \mmat{D}(\cdot) $ denotes the diagonal part and $ \mmat{L} (\cdot) $ the strict lower triangle of a
given matrix). This can be written in the form \eqref{Eq:CompactRepresentation} by defining
\begin{equation}\label{Eq:CompactRepresentationDefBFGS}
    \mmat{A_k}:=
\begin{bmatrix}
    \mmat{S}_k & \mmat{Y}_k
\end{bmatrix}
    \quad\text{and}\quad
    \mmat{Q}_k:=
\begin{bmatrix}
    -\gamma_k^{-1} \mmat{S}_k^{\Trp} \mmat{S}_k & -\gamma_k^{-1} \mmat{L}_k \\
    -\gamma_k^{-1} \mmat{L}_k^{\Trp} & \mmat{D}_k
\end{bmatrix}
    .
\end{equation}
Note that $\mmat{Q}_k\in\R^{2 m\times 2 m}$.

The BFGS formula has a significant advantage in that the well-definedness of the updates can be controlled.
More specifically, assuming that $\mvec{s}_k^{\Trp}\mvec{y}_k>0$ for all $k$, it can be shown that the
BFGS matrix $\mmat{B}_k$ is positive definite, so that the regularized BFGS matrix
$\hat{\mmat{B}}_k=\mmat{B}_k+\RegParam\mmat{I}$ is also positive definite and therefore nonsingular.
By the Woodbury matrix identity, this implies that the inner matrix
$\mmat{Q}_k+\hat{\gamma}_k^{-1} \mmat{A}_k^{\Trp} \mmat{A}_k$ in \eqref{Eq:RegularizedQNinverse_I} is
invertible, and thus the regularized Newton step is well-defined for all $\RegParam\ge 0$.

In practice, the well-definedness is controlled by means of a so-called \emph{cautious updating} mechanism \cite{Li2001}. The previous limited memory data is only updated with the next pair $(\mvec{s}_k,\mvec{y}_k)$ if
\begin{equation}\label{Eq:CautiousUpdateLBFGS}
    \mvec{y}_k^{\Trp}\mvec{s}_k \ge \varepsilon \|\mvec{s}_k\|^2,
\end{equation}
where $\varepsilon>0$ is some predefined constant. This guarantees that the L-BFGS matrices $\mmat{B}_k$ are uniformly positive definite. If $\nabla f$ is Lipschitz continuous on the set of iterates (or an appropriate level set), then \eqref{Eq:CautiousUpdateLBFGS} also guarantees that $\{\mmat{B}_k\}$ is bounded.

\subsubsection*{Updating L-BFGS information}

We now describe how the L-BFGS information can be updated in an efficient manner. To avoid repetition, we only
describe the case where the previous information is already ``full'', i.e., where at least $m$ previous data pairs $(\mvec{s}_i,\mvec{y}_i)$ are available. The modifications necessary to treat the initial steps essentially amount to re-indexing and will not be detailed here.

Much of the computational effort of regularized L-BFGS can be mitigated by memorizing certain intermediate results.
Motivated by a related trust-region approach in \cite{Burdakov2017}, we track, in addition to the matrices $\mmat{S}_k$ and $\mmat{Y}_k$, the quantities
\begin{equation*}
    \mmat{A}_k^{\Trp}\mmat{A}_k \in \R^{2 m\times 2 m}
    \quad\text{and}\quad
    \mmat{A}_k^{\Trp} \mvec{g}_k \in \R^{2 m}.
\end{equation*}
Both of these quantities are necessary for the computation of the regularized quasi-Newton
step \eqref{Eq:RegularizedQNstep}, \eqref{Eq:RegularizedQNstep_p},
but they also occur in other places of the iteration and updating process,
so that memorizing them can save redundant computational effort. Recall that
$\mmat{A}_k=[\mmat{S}_k \, , \, \mmat{Y}_k]$, so that in particular
\begin{equation*}
    \mmat{A}_k^{\Trp}\mmat{A}_k=
\begin{bmatrix}
    \mmat{S}_k^{\Trp} \mmat{S}_k & \mmat{S}_k^{\Trp} \mmat{Y}_k \\[1pt]
    \mmat{Y}_k^{\Trp} \mmat{S}_k & \mmat{Y}_k^{\Trp} \mmat{Y}_k
\end{bmatrix}
    \quad\text{and}\quad \mmat{A}_k^{\Trp}\mvec{g}_k=
\begin{bmatrix}
    \mmat{S}_k^{\Trp} \mvec{g}_k \\[1pt]
    \mmat{Y}_k^{\Trp} \mvec{g}_k
\end{bmatrix}.
\end{equation*}
Hence, the matrix $\mmat{A}_k^{\Trp}\mmat{A}_k$ contains the blocks $\mmat{S}_k^{\Trp}\mmat{S}_k$, $\mmat{L}_k$,
and $\mmat{D}_k$ from \eqref{Eq:CompactRepresentationDefBFGS} as submatrices.

When passing from $k$ to $k+1$, these matrices and vectors can be updated as follows. If the data pair
$(\mvec{s}_k,\mvec{y}_k)$ is rejected, then $\mmat{A}_k$ remains unchanged, and we may update
$\mmat{A}_k^{\Trp}\mvec{g}_k$ by direct computation. If the data pair is accepted, then the updating process
requires more care since both $\mmat{A}_k^{\Trp}\mmat{A}_k$ and $\mmat{A}_k^{\Trp}\mvec{g}_k$ need to be
incremented. In this case, the new matrices $ \mmat{S}_{k+1} $ and $ \mmat{Y}_{k+1} $ consist of the last
$ m-1 $ columns of the old matrices $ \mmat{S}_{k} $ and $ \mmat{Y}_{k} $, respectively, to which the
new vectors $\mvec{s}_k$ and $\mvec{y}_k$ are appended in the last column. We then begin by computing the vectors
\begin{equation}
    \mvec{v}:= \mmat{A}_k^{\Trp} \mvec{s}_k=-\hat{\gamma}_k^{-1} \mmat{A}_k^{\Trp} \mvec{g}_k+\hat{\gamma}_k^{-2} (\mmat{A}_k^{\Trp}\mmat{A}_k) \mvec{p}_k, \qquad
    \mvec{w}:= \mmat{A}_{k+1}^{\Trp} \mvec{g}_{k+1},
\end{equation}
where $\mvec{p}_k$ is given by \eqref{Eq:RegularizedQNstep_p}; as well as the scalar quantities $(\alpha_1, \alpha_2, \alpha_3) := (\mvec{s}_k^{\Trp} \mvec{s}_k, \mvec{s}_k^{\Trp} \mvec{y}_k, \mvec{y}_k^{\Trp} \mvec{y}_k)$.
This information is then used to update $\mmat{A}_k^{\Trp} \mmat{A}_k$ blockwise using the formulas
\begin{subequations}
\begin{align}
    \mmat{S}_{k+1}^{\Trp} \mmat{S}_{k+1} & =
\begin{bmatrix}
    (\mmat{S}_k^{\Trp}\mmat{S}_k)_{2:m, 2:m} & \mvec{v}_{2:m} \\
    * & \alpha_1
\end{bmatrix},\\[3pt]
    \mmat{S}_{k+1}^{\Trp} \mmat{Y}_{k+1} & =
\begin{bmatrix}
    (\mmat{S}_k^{\Trp} \mmat{Y}_k)_{2:m,2:m} & \mvec{w}_{1:m-1}-(\mmat{A}_k^{\Trp}\mvec{g}_k)_{2:m} \\[3pt]
    \mvec{v}_{m+2:2 m}^{\Trp} & \alpha_2
\end{bmatrix}, \\[3pt]
    \mmat{Y}_{k+1}^{\Trp} \mmat{Y}_{k+1} & =
\begin{bmatrix}
    (\mmat{Y}_k^{\Trp} \mmat{Y}_k)_{2:m,2:m} & \mvec{w}_{m+1:2 m-1}-(\mmat{A}_k^{\Trp} \mvec{g}_k)_{m+2:2 m} \\[3pt]
    * & \alpha_3
\end{bmatrix},
\end{align}
\end{subequations}
where ``$*$'' is given by symmetry, and expressions of the form $(\mmat{S}_k^{\Trp}\mmat{S}_k)_{2:m,2:m}$ or $\mvec{v}_{2:m}$ denote submatrices and -vectors built from the subscripted index ranges. Finally, we have $\mmat{Y}_{k+1}^{\Trp}\mmat{S}_{k+1}=(\mmat{S}_{k+1}^{\Trp}\mmat{Y}_{k+1})^{\Trp}$, and the new vector $\mmat{A}_{k+1}^{\Trp} \mvec{g}_{k+1}$ is by definition equal to $\mvec{w}$.

\subsubsection*{Computational complexity}

Let us now comment on the complexity involved in the computation of the regularized quasi-Newton step. Assuming that the product $\mmat{A}_k^{\Trp}\mvec{g}_k$ has been formed, the main cost is the solution of a $2 m\times 2 m$ symmetric linear system to form $\mvec{p}_k$, and the multiplication of $\mvec{p}_k$ with the $n\times 2 m$ matrix $\mmat{A}_k$. Hence, the complexity of the regularized quasi-Newton equation is $2 m n+O(m^3)$ multiplications.

When a step is successful, the existing data needs to be updated according to the formulas developed above. The dominating cost of this is $2 m n$ multiplications for the computation of $\mvec{w}=\mmat{A}_{k+1}^{\Trp}\mvec{g}_{k+1}$. Hence, the overall computational effort is at most $2 m n$ multiplications for an unsuccessful step, and $4 m n$ for a successful step.

The computational cost of the $2m \times 2m$ linear equation \eqref{Eq:RegularizedQNstep_p} for the computation of $\mvec{p}_k$ is of order $O(m^3)$. Thus, if $m\ll n$, this cost is negligible in comparison to $m n$. The slight computational overhead induced by this linear equation can be mitigated further by using the Schur complement of $\mmat{Q}_k + \hat{\gamma}_k^{-1}\mmat{A}_k^{\Trp}\mmat{A}_k$ to reduce the $2m \times 2m$ inversion to two $m\times m$ Cholesky factorizations. See \cite{Burke2008} for more details.

\subsection{Symmetric rank-one (SR1)}\label{Sec:Matrices:SR1}

For SR1, the compact representation takes on the form
\begin{equation}\label{Eq:CompactRepresentationSR1}
    \mmat{B}_k=\InitialB+(\mmat{Y}_k-\InitialB \mmat{S}_k)
    (\mmat{D}_k+\mmat{L}_k+\mmat{L}_k^{\Trp}-\mmat{S}_k^{\Trp}\InitialB\mmat{S}_k)^{-1}
    (\mmat{Y}_k-\InitialB \mmat{S}_k)^{\Trp},
\end{equation}
    where $\mmat{D}_k$ and $\mmat{L}_k$ are again given by \eqref{Eq:AuxiliaryMatricesQN}. This can be written in the form \eqref{Eq:CompactRepresentation} by defining
\begin{equation}\label{Eq:CompactRepresentationDefSR1}
    \mmat{A}_k:=\mmat{Y}_k -\InitialB \mmat{S}_k
    \quad\text{and}\quad
    \mmat{Q}_k:=\mmat{D}_k+\mmat{L}_k+\mmat{L}_k^{\Trp}-\mmat{S}_k^{\Trp}\InitialB\mmat{S}_k.
\end{equation}
Note that $\mmat{Q}_k\in\R^{m\times m}$ in this case.

If $\InitialB=\gamma_k \mmat{I}$, then \eqref{Eq:CompactRepresentationSR1} can be simplified to
\begin{equation}\label{Eq:CompactRepresentationBFGS_I}
    \mmat{B}_k =  \gamma_k \mmat{I} + (\mmat{Y}_k - \gamma_k \mmat{S}_k)(\mmat{D}_k + \mmat{L}_k + \mmat{L}_k^{\Trp} - \gamma_k \mmat{S}_k^{\Trp}\mmat{S}_k)^{-1}(\mmat{Y}_k - \gamma_k \mmat{S}_k)^{\Trp}.
\end{equation}
The well-definedness of the SR1 update is hard to guarantee in practice because the underlying rank one formula involves a denominator of the form $(\mvec{y}_k-\mmat{B}_k\mvec{s}_k)^{\Trp}\mvec{s}_k$, which can vanish. Thus, when applying formula \eqref{Eq:RegularizedQNinverse} to the SR1 setting, it is important to clarify how this situation is handled. Note that it is not possible to predict which new data $(\mvec{s}_{k+1},\mvec{y}_{k+1})$ might lead to ill-conditioning because this crucially depends on the previous information $(\mmat{S}_k,\mmat{Y}_k)$. In fact, even the discarding of old data at some point during the iteration might have an influence and change the well-definedness of the SR1 update.

Fortunately, there is a simple and effective way of skipping ill-conditioned updates ``on the fly'', i.e., during the computation of the quasi-Newton step. This effectively amounts to skipping an intermediate step $(\mvec{s}_i,\mvec{y}_i)$ when necessary and proceeding the SR1 update with $(\mvec{s}_{i+1},\mvec{y}_{i+1})$ instead. It was observed in \cite{Byrd1994} that ill-definedness of one of these updates amounts to the singularity of a principal minor of $\mmat{Q}_k$, or equivalently, to a vanishing pivot element during a triangularization of $\mmat{Q}_k$. When this occurs, it is proposed in \cite{Byrd1994} to skip the update by essentially ignoring the current row and column of $\mmat{Q}_k$, and the current column of $\mmat{A}_k$ (which contains the corresponding vectors $\mvec{s}_i$ and $\mvec{y}_i$).

The above procedure can be adapted to the \emph{regularized} SR1 setting by observing that the SR1 update ``commutes'' with the regularization in a certain sense. More specifically, if $\mmat{B}_k=\operatorname{SR1}(\InitialB,\mmat{S},\mmat{Y})$ denotes the SR1 update, then
\begin{equation*}
    \operatorname{SR1}(\InitialB+\RegParam \mmat{I},\mmat{S},\mmat{Y}+\RegParam \mmat{S})=
    \operatorname{SR1}(\InitialB,\mmat{S},\mmat{Y})+\RegParam \mmat{I}
\end{equation*}
for all $\RegParam\ge 0$, provided that the left side exists. Moreover, an easy calculation shows that the matrix $\mmat{Q}_k+\mmat{A}_k^{\Trp} \InitialBhat^{-1} \mmat{A}_k$ from \eqref{Eq:RegularizedQNinverse}, which needs to be inverted for the computation of the regularized Newton step, coincides (up to scaling) with the analogue of $\mmat{Q}_k$ which would arise for the SR1 update corresponding to $\InitialBhat$ and $\mmat{Y}_k+\RegParam \mmat{S}_k$.

\subsubsection*{Updating L-SR1 information}

The quantities involved in the L-SR1 computations can be updated in a similar fashion to the L-BFGS case; see Section~\ref{Sec:Matrices:BFGS}. We again maintain the quantities
\begin{equation}\label{Eq:SR1quantities}
    \mmat{S}_k^{\Trp}\mmat{S}_k,\,\mmat{S}_k^{\Trp}\mmat{Y}_k,\,
    \mmat{Y}_k^{\Trp}\mmat{Y}_k\in\R^{m\times m}
    \quad\text{and}\quad
    \mmat{S}_k^{\Trp}\mvec{g}_k,\,\mmat{Y}_k^{\Trp}\mvec{g}_k\in\R^m.
\end{equation}
These can be formed and updated as before. Moreover, they can be used to directly form the matrices $\mmat{A}_k$ and $\mmat{Q}_k$, the product $\mmat{A}_k^{\Trp}\mvec{g}_k$, and the matrix $\mmat{A}_k^{\Trp}\mmat{A}_k$.

\subsubsection*{Computational complexity}

The computational cost of the regularized L-SR1 method is as follows. In each successful iteration, the quantities \eqref{Eq:SR1quantities} are updated, and the matrix $\mmat{A}_k=\mmat{Y}_k-\InitialB\mmat{S}_k$ is formed. Using the techniques from Section~\ref{Sec:Matrices:BFGS}, these operations require $3 m n$ multiplications.

Moreover, the quasi-Newton step needs to be calculated in each step, which entails the solution of an $m\times m$ symmetric linear system to obtain $\mvec{p}_k$, and the multiplication of $\mvec{p}_k$ with the $n\times m$ matrix $\mmat{A}_k$, requiring another $m n$ multiplications.

In total, the cost of a successful step is therefore $4 m n$ multiplications, and the cost of an unsuccessful step is $m n$ multiplications (down from $2 m n$ in the BFGS case).

\subsection{Powell-symmetric-Broyden (PSB)}\label{Sec:Matrices:PSB}

As a third example, we include the classical PSB formula from \cite{Powell1970}.
This approach is interesting because the PSB update is always well-defined and has certain well-known minimality properties. The PSB update is given by
\begin{equation}\label{Eq:PSB}
    \mmat{B}_{k+1} = \mmat{B}_k+\frac{(\mvec{y}_k-\mmat{B}_k \mvec{s}_k)\mvec{s}_k^{\Trp}+\mvec{s}_k(\mvec{y}_k-\mmat{B}_k \mvec{s}_k)^{\Trp}}{\mvec{s}_k^{\Trp}\mvec{s}_k}-
    \frac{(\mvec{y}_k-\mmat{B}_k \mvec{s}_k)^{\Trp}\mvec{s}_k}{(\mvec{s}_k^{\Trp}\mvec{s}_k)^2} \mvec{s}_k \mvec{s}_k^{\Trp}.
\end{equation}
The compact representation of PSB is given in the next theorem.

Note that there is a related representation in \cite{Burdakov2002} for a multipoint secant version of PSB. The two representations coincide when $m=1$.

\begin{theorem}[Compact representation of PSB]\label{thm:PSB}
    The PSB formula admits the compact representation
\begin{equation}\label{Eq:PSBcompact}
    \mmat{B}_k=\InitialB+
\begin{bmatrix}
    \mmat{S}_k & \mmat{W}_k
\end{bmatrix}
\begin{bmatrix}
    0 & \mmat{U}_k \\
    \mmat{U}_k^{\Trp} & \mmat{L}_k+\mmat{D}_k+\mmat{L}_k^{\Trp}
\end{bmatrix}^{-1}
\begin{bmatrix}
    \mmat{S}_k & \mmat{W}_k
\end{bmatrix}^{\Trp},
\end{equation}
    where $\mmat{W}_k:=\mmat{Y}_k-\InitialB\mmat{S}_k$, $\mmat{U}_k$ is the (non-strictly) upper triangular part of $\mmat{S}_k^{\Trp}\mmat{S}_k$, $\mmat{L}_k$ is the strictly lower triangular part of $\mmat{S}_k^{\Trp}\mmat{W}_k$, and $\mmat{D}_k$ is the diagonal part of $\mmat{S}_k^{\Trp}\mmat{W}_k$.
\end{theorem}

\begin{proof}
To simplify some technical details, we restrict the proof to the case where $k = m$ (i.e., the algorithm has performed exactly $m$ steps, and the matrices $\mmat{S}_k$ and $\mmat{Y}_k$ are ``full''). Observe first that \eqref{Eq:PSB} can be rewritten as
\begin{equation*}
    \mmat{B}_{k+1}= \mleft( \mmat{I}-\frac{\mvec{s}_k \mvec{s}_k^{\Trp}}{\mvec{s}_k^{\Trp}\mvec{s}_k} \mright) \mmat{B}_k \mleft( \mmat{I}-\frac{\mvec{s}_k \mvec{s}_k^{\Trp}}{\mvec{s}_k^{\Trp}\mvec{s}_k} \mright)+
\begin{bmatrix}
    \mvec{s}_k & \mvec{y}_k
\end{bmatrix}
\begin{bmatrix}
    0 & \mvec{s}_k^{\Trp}\mvec{s}_k \\
    \mvec{s}_k^{\Trp}\mvec{s}_k & \mvec{s}_k^{\Trp}\mvec{y}_k
\end{bmatrix}^{-1}
\begin{bmatrix}
    \mvec{s}_k & \mvec{y}_k
\end{bmatrix}^{\Trp}.
\end{equation*}
    Therefore, we can write $\mmat{B}_k=\mmat{M}_k+\mmat{N}_k$, where $\mmat{M}_k,\mmat{N}_k$ are recursively defined through the formulas
\begin{align*}
    \mmat{M}_0=\InitialB, \qquad & \mmat{M}_{i+1}=\mmat{V}_i \mmat{M}_i \mmat{V}_i,\\
    \mmat{N}_0=0, \qquad & \mmat{N}_{i+1}=\mmat{V}_i \mmat{N}_i \mmat{V}_i +
\begin{bmatrix}
    \mvec{s}_i & \mvec{y}_i
\end{bmatrix}
\begin{bmatrix}
    0 & \mvec{s}_i^{\Trp}\mvec{s}_i \\
    \mvec{s}_i^{\Trp}\mvec{s}_i & \mvec{s}_i^{\Trp}\mvec{y}_i
\end{bmatrix}^{-1}
\begin{bmatrix}
    \mvec{s}_i & \mvec{y}_i
\end{bmatrix}^{\Trp},
\end{align*}
    where $\mmat{V}_i:=\mmat{I}-(\mvec{s}_i^{\Trp}\mvec{s}_i)^{-1}\mvec{s}_i\mvec{s}_i^{\Trp}$ for all $i$.
    Observe now that $\mmat{V}_0 \cdot \ldots \cdot \mmat{V}_{k-1}=\mmat{I}-\mmat{S}_k \mmat{U}_k^{-1}\mmat{S}_k^{\Trp}$
    by \cite[Lem.~2.1]{Byrd1994}, so that
\begin{equation*}
    \mmat{M}_k= \big(\mmat{I}-\mmat{S}_k \mmat{U}_k^{-\Trp}\mmat{S}_k^{\Trp} \big) \InitialB
    \big( \mmat{I}-\mmat{S}_k \mmat{U}_k^{-1}\mmat{S}_k^{\Trp} \big).
\end{equation*}
    We proceed by using (finite) induction to show that
\begin{equation}\label{eq:induction}
    \mmat{N}_i=\mmat{S}_i \mmat{U}_i^{-\Trp} \mmat{Y}_i^{\Trp}+\mmat{Y}_i \mmat{U}_i^{-1} \mmat{S}_i^{\Trp}-
    \mmat{S}_i \mmat{U}_i^{-\Trp} \big( \tilde{\mmat{L}}_i+\tilde{\mmat{D}}_i+\tilde{\mmat{L}}_i^{\Trp} \big)
    \mmat{U}_i^{-1} \mmat{S}_i^{\Trp} \quad \text{for all } i = 1, \ldots, k,
\end{equation}
    where $\tilde{\mmat{L}}_i:=\mmat{L}(\mmat{S}_i^{\Trp}\mmat{Y}_i)$ and
    $ \tilde{\mmat{D}}_i := \mmat{D} (\mmat{S}_i^{\Trp}\mmat{Y}_i)$. Before we verify this formula,
    we show that it yields the desired compact representation of the PSB formula. Indeed, using
    \eqref{eq:induction} and the definitions of the matrices $ \tilde{\mmat{L}}_k, \tilde{\mmat{D}}_k, \mmat{L}_k, \mmat{D}_k $, respectively, we obtain
\begin{eqnarray*}
    \mmat{B}_k & = & \mmat{M}_k + \mmat{N}_k \\
    & = & \InitialB - \mmat{S}_k \mmat{U}_k^{-\Trp} \mmat{S}_k^{\Trp}
    \InitialB - \InitialB \mmat{S}_k \mmat{U}_k^{-1} \mmat{S}_k^{\Trp} +
    \mmat{S}_k \mmat{U}_k^{-\Trp} \mmat{Y}_k^{\Trp} + \mmat{Y}_k \mmat{U}_k^{-1} \mmat{S}_k^{\Trp} \\
    & & - \mmat{S}_k \mmat{U}_k^{-\Trp} \big( \mmat{L}_k + \mmat{D}_k + \mmat{L}_k^{\Trp} \big)
    \mmat{U}_k^{-1} \mmat{S}_k^{\Trp}.
\end{eqnarray*}
On the other hand, exploiting the fact that
\begin{equation*}
   \begin{bmatrix}
    0 & \mmat{U}_k \\
    \mmat{U}_k^{\Trp} & \mmat{L}_k+\mmat{D}_k+\mmat{L}_k^{\Trp}
    \end{bmatrix}^{-1}
    =
    \begin{bmatrix}
    - \mmat{U}_k^{-\Trp} \big( \mmat{L}_k+\mmat{D}_k+\mmat{L}_k^{\Trp} \big) \mmat{U}_k^{-1} &
    \mmat{U}_k^{-\Trp} \\
    \mmat{U}_k^{-1} & 0
    \end{bmatrix},
\end{equation*}
using $ \mmat{W}_k=\mmat{Y}_k-\InitialB\mmat{S}_k $, and expanding \eqref{Eq:PSBcompact},
it is easy to see that we obtain the same expression.

Hence it remains to verify \eqref{eq:induction} by induction. For $ i = 1 $, we have
\begin{equation*}
   \mmat{S}_1 = \begin{bmatrix} s_0 \end{bmatrix}, \quad
   \mmat{Y}_1 = \begin{bmatrix} y_0 \end{bmatrix}, \quad
   \mmat{U}_1^{-1} = \frac{1}{\mvec{s}_0^{\Trp} \mvec{s}_0}, \quad
   \tilde{\mmat{L}}_1 = \begin{bmatrix} 0 \end{bmatrix}, \quad
   \tilde{\mmat{D}}_1 = \mvec{s}_0^{\Trp} \mvec{y_0}.
\end{equation*}
Together with the observation that
\begin{equation}\label{eq:sinverse}
   \begin{bmatrix}
   0 & \mvec{s}_i^{\Trp} \mvec{s}_i \\
   \mvec{s}_i^{\Trp} \mvec{s}_i & \mvec{s}_i^{\Trp} \mvec{y}_i
   \end{bmatrix}^{-1}
   =
   \begin{bmatrix}
   - \frac{\mvec{s}_i^{\Trp} \mvec{y}_i}{( \mvec{s}_i^{\Trp} \mvec{s}_i )^2} &
   \frac{1}{\mvec{s}_i^{\Trp} \mvec{s}_i} \\
   \frac{1}{\mvec{s}_i^{\Trp} \mvec{s}_i} & 0
   \end{bmatrix},
\end{equation}
an elementary calculation shows that \eqref{eq:induction} holds for $ i = 1 $.
Suppose the statement is true for some $ i = 1,\ldots,k-1$.
Using the induction hypothesis together with \eqref{eq:sinverse}, a straightforward calculation shows that
\begin{eqnarray*}
   \mmat{N}_{i+1} & = & \mmat{V}_i \mmat{S}_i \mmat{U}_i^{-\Trp} \mmat{Y}_i^{\Trp} \mmat{V}_i +
   \mmat{V}_i \mmat{Y}_i \mmat{U}_i^{-1} \mmat{S}_i^{\Trp} \mmat{V}_i -
   \mmat{V}_i \mmat{S}_i \mmat{U}_i^{- \Trp} \big(
   \tilde{\mmat{L}}_i+\tilde{\mmat{D}}_i+\tilde{\mmat{L}}_i^{\Trp} \big)
   \mmat{U}_i^{-1} \mmat{S}_i^{\Trp} \mmat{V}_i \\
   & & - \frac{\mvec{s}_i^{\Trp} \mvec{y}_i}{( \mvec{s}_i^{\Trp} \mvec{s}_i )^2} \mvec{s}_i \mvec{s}_i^{\Trp} +
   \frac{1}{\mvec{s}_i^{\Trp} \mvec{s}_i} \mvec{s}_i \mvec{y}_i^{\Trp} +
   \frac{1}{\mvec{s}_i^{\Trp} \mvec{s}_i} \mvec{y}_i \mvec{s}_i^{\Trp}.
\end{eqnarray*}
On the other hand, let us calculate the expression \eqref{eq:induction} for $ i + 1 $. Based on the
partitions
\begin{eqnarray*}
   & & \mmat{S}_{i+1} = \begin{bmatrix} \mmat{S}_i & \mvec{s}_i \end{bmatrix}, \\
   & & \mmat{Y}_{i+1} = \begin{bmatrix} \mmat{Y}_i & \mvec{y}_i \end{bmatrix}, \\
   & & \mmat{U}_{i+1} = \begin{bmatrix} \mmat{U}_i & \mmat{S}_i^{\Trp} \mvec{s}_i \\ 0 &
   \mvec{s}_i^{\Trp} \mvec{s}_i \end{bmatrix} \quad \Longrightarrow \quad
   \mmat{U}_{i+1}^{-1} = \begin{bmatrix} \mmat{U}_i^{-1} & - \frac{1}{\mvec{s}_i^{\Trp} \mvec{s}_i}
   \mmat{U}_i^{-1} \mmat{S}_i^{\Trp} \mvec{s}_i \\ 0 & \frac{1}{\mvec{s}_i^{\Trp} \mvec{s}_i} \end{bmatrix}, \\
   & & \tilde{\mmat{L}}_{i+1} = \begin{bmatrix} \tilde{\mmat{L}}_i & 0 \\ \mvec{s}_i^{\Trp} \mmat{Y}_i & 0
   \end{bmatrix}, \\
   & & \tilde{\mmat{D}}_{i+1} = \begin{bmatrix} \tilde{\mmat{D}}_i & 0 \\ 0 & \mvec{s}_i^{\Trp} \mvec{y}_i
   \end{bmatrix},
\end{eqnarray*}
we obtain
\begin{eqnarray*}
   \mmat{S}_{i+1} \mmat{U}_{i+1}^{-\Trp} & = & \begin{bmatrix} \mmat{V}_i \mmat{S}_i \mmat{U}_i^{-\Trp} &
   \frac{1}{\mvec{s}_i^{\Trp} \mvec{s}_i} \mvec{s}_i \end{bmatrix}, \\
   \mmat{S}_{i+1} \mmat{U}_{i+1}^{-\Trp} \mmat{Y}_{i+1}^{\Trp} & = & \mmat{V}_i \mmat{S}_i \mmat{U}_i^{-\Trp}
   \mmat{Y}_i^{\Trp} + \frac{1}{\mvec{s}_i^{\Trp} \mvec{s}_i} \mvec{s}_i \mvec{y}_i^{\Trp}, \\
   \tilde{\mmat{L}}_{i+1} + \tilde{\mmat{D}}_{i+1} + \tilde{\mmat{L}}_{i+1}^{\Trp} & = &
   \begin{bmatrix} \tilde{\mmat{L}}_{i} + \tilde{\mmat{D}}_{i} + \tilde{\mmat{L}}_{i}^{\Trp} &
   \mmat{Y}_i^{\Trp} \mvec{s}_i \\ \mvec{s}_i^{\Trp} \mmat{Y}_i & \mvec{s}_i^{\Trp} \mvec{y}_i
   \end{bmatrix},
\end{eqnarray*}
hence
\begin{eqnarray*}
   \lefteqn{\mmat{S}_{i+1} \mmat{U}_{i+1}^{-\Trp} \big( \tilde{\mmat{L}}_{i+1} + \tilde{\mmat{D}}_{i+1} +
   \tilde{\mmat{L}}_{i+1}^{\Trp} \big) \mmat{U}_{i+1}^{-1} \mmat{S}_{i+1}^{\Trp}} \\
   & = & \mmat{V}_i \mmat{S}_i \mmat{U}_i^{-\Trp} \big( \tilde{\mmat{L}}_{i} + \tilde{\mmat{D}}_{i} +
   \tilde{\mmat{L}}_{i}^{\Trp} \big) \mmat{U}_i^{-1} \mmat{S}_i^{\Trp} \mmat{V}_i +
   \frac{1}{\mvec{s}_i^{\Trp} \mvec{s}_i} \mmat{V}_i \mmat{S}_i \mmat{U}_i^{-\Trp} \mmat{Y}_i^{\Trp}
   \mvec{s}_i \mvec{s}_i^{\Trp} \\
   & & + \frac{1}{\mvec{s}_i^{\Trp} \mvec{s}_i} \mvec{s}_i \mvec{s}_i^{\Trp} \mmat{Y}_i
   \mmat{U}_i^{-1} \mmat{S}_i^{\Trp} \mmat{V}_i +
   \frac{\mvec{s}_i^{\Trp} \mvec{y}_i}{( \mvec{s}_i^{\Trp} \mvec{s}_i )^2} \mvec{s}_i \mvec{s}_i^{\Trp}.
\end{eqnarray*}
Using these expressions and expanding \eqref{eq:induction} with $ i $ replaced by $ i+1 $, and
taking into account once again the definition of
$ \mmat{V}_i $, an elementary calculation shows that the resulting matrix $ \mmat{N}_{i+1} $
coincides with the one obtained previously. This completes the induction.
\end{proof}

If $\InitialB =\gamma_k \mmat{I}$ for some $\gamma_k\in\R$, then \eqref{Eq:PSBcompact} can be rewritten as
\begin{equation}\label{Eq:PSBcompactI}
    \mmat{B}_k=\gamma_k \mmat{I}+\mmat{A}_k
\begin{bmatrix}
    0 & \mmat{U}_k \\
    \mmat{U}_k^{\Trp} & \mmat{D}(\mmat{S}_k^{\Trp}\mmat{Y}_k) + \gamma_k \mmat{D}(\mmat{S}_k^{\Trp}\mmat{S}_k)+ \mmat{L}(\mmat{S}_k^{\Trp}\mmat{Y}_k)+\mmat{L}(\mmat{S}_k^{\Trp}\mmat{Y}_k)^{\Trp}
\end{bmatrix}^{-1}
    \mmat{A}_k^{\Trp},
\end{equation}
where $\mmat{A}_k=[\mmat{S}_k, \mmat{Y}_k]$ as before. This form of $\mmat{B}_k$ has the advantage that all involved quantities can be obtained as submatrices of the product $\mmat{A}_k^{\Trp}\mmat{A}_k$.

\subsubsection*{Updating and Complexity}

As before, the L-PSB quantities can be updated in a similar fashion to the L-BFGS case; see Section~\ref{Sec:Matrices:BFGS}. We again maintain the quantities
\begin{equation}\label{Eq:PSBquantities}
    \mmat{S}_k^{\Trp}\mmat{S}_k,\,\mmat{S}_k^{\Trp}\mmat{Y}_k,\,
    \mmat{Y}_k^{\Trp}\mmat{Y}_k\in\R^{m\times m}
    \quad\text{and}\quad
    \mmat{S}_k^{\Trp}\mvec{g}_k,\,\mmat{Y}_k^{\Trp}\mvec{g}_k\in\R^m.
\end{equation}
These can be updated as before and used to compute the quasi-Newton direction via the inverse formula \eqref{Eq:RegularizedQNinverse}. The complexity of the L-PSB step equals that of L-BFGS.

\section{Numerical Experiments}\label{Sec:Numerics}

The benchmark implementation described here can be found online at \url{https://github.com/dmsteck/paper-regularized-qn-benchmark}.

In this section, we compare a selection of regularized quasi-Newton methods (Algorithm~\ref{Alg:RegularizedQN}) amongst each other and with existing L-BFGS type line search and trust region algorithms from the literature.

Algorithms were tested on all large-scale ($n\ge 1000$) problems from the \texttt{CUTEst} collection \cite{Gould2015}. The implementation was done in \texttt{Python3} using the \texttt{PyCUTEst} interface \cite{Fowkes2019}. All problems were computed with initial points as supplied by the library. We excluded test problems where all algorithms failed within the threshold of 100,000 iterations (see below). We also omitted \texttt{FLETCBV2} because the initial point is a stationary point. The final test set after these considerations consists of 77 problems.

The results for different algorithms are compared using performance profiles \cite{Dolan2002} based on the number of function evaluations. Note that the regularization methods evaluate the function exactly once per successful or unsuccessful step, so that the number of function evaluations equals the number of iterations. Furthermore, aside from function or gradient evaluations, all tested methods have a similar computational complexity per step (see \cite{Burdakov2017,Liu1989} and Section~\ref{Sec:Matrices}), so that function evaluations provide a simple yet meaningful baseline metric.

Note that we didn't account for gradient evaluations in our comparison; the regularization methods (and the trust-region comparison method in Section~\ref{Sec:Numerics:Existing}) evaluate $\nabla f$ exactly once in every successful iteration whereas Wolfe-based line search methods evaluate $\nabla f$ within the inner line search loop. Hence, accounting for gradient evaluations would benefit many of our methods in the subsequent comparisons. However, to keep things simple, we have avoided a more granular breakdown and focused exclusively on function evaluations.

Whenever an algorithm didn't solve a particular problem to within tolerance (see below), the number of function evaluations was set to $+\infty$ for the purpose of comparison.

\subsection{Comparison of Regularized Limited Memory Methods}\label{Sec:Numerics:Reg}

We implemented the following four regularization-based algorithms:
\begin{itemize}[leftmargin=8em]
\item[{\bfseries\sffamily regLBFGS:}] Algorithm~\ref{Alg:RegularizedQN} using the L-BFGS technique as set out in Section~\ref{Sec:Matrices:BFGS};
\item[{\bfseries\sffamily regLBFGSsec:}] Algorithm~\ref{Alg:RegularizedQN} using the regularized secant version of L-BFGS as discussed in Remark~\ref{Rem:ModifiedSecant} (see also \cite{Sugimoto2014});
\item[{\bfseries\sffamily regLSR1:}] Algorithm~\ref{Alg:RegularizedQN} using the L-SR1 technique as set out in Section~\ref{Sec:Matrices:SR1};
\item[{\bfseries\sffamily regLPSB:}] Algorithm~\ref{Alg:RegularizedQN} using the L-PSB technique as set out in Section~\ref{Sec:Matrices:PSB}.
\end{itemize}
The implementations all use the same hyperparameters
\begin{equation}\label{Eq:Hyperparameters}
m=5, \quad
\RegParam_0=1, \quad
p_{\min}=c_1=10^{-4}, \quad
c_2=0.9, \quad
\sigma_1 = 0.5, \quad
\sigma_2 = 4.
\end{equation}
To guarantee well-definedness, \textsf{regLBFGS} and \textsf{regLBFGSsec} are implemented using the cautious updating scheme \eqref{Eq:CautiousUpdateLBFGS} with $\varepsilon:=10^{-8}$. The \textsf{regLSR1} and \textsf{regLPSB} algorithms benefit from indefinite Hessian approximations \cite{Byrd1996,Conn1988} and therefore were not combined with the cautious updating scheme. However, for these methods, the cautious scheme was still applied to the update of the rolling initial approximation \eqref{Eq:InitialB}; see below.

Inspired by a technique from \cite{Burdakov2017}, all algorithms begin with a single Mor\'e--Thuente line search along the normalized negative gradient direction prior to the main iteration loop (see Section~\ref{Sec:Numerics:Existing} for more details). This has the advantage of providing an initial memory pair $(\mvec{s}_0, \mvec{y}_0)$ that passes the cautious update check \eqref{Eq:CautiousUpdateLBFGS}, and reducing the impact of any initial backtracking on the iteration numbers.

\begin{figure}\centering
    \includegraphics[width=0.45\textwidth]{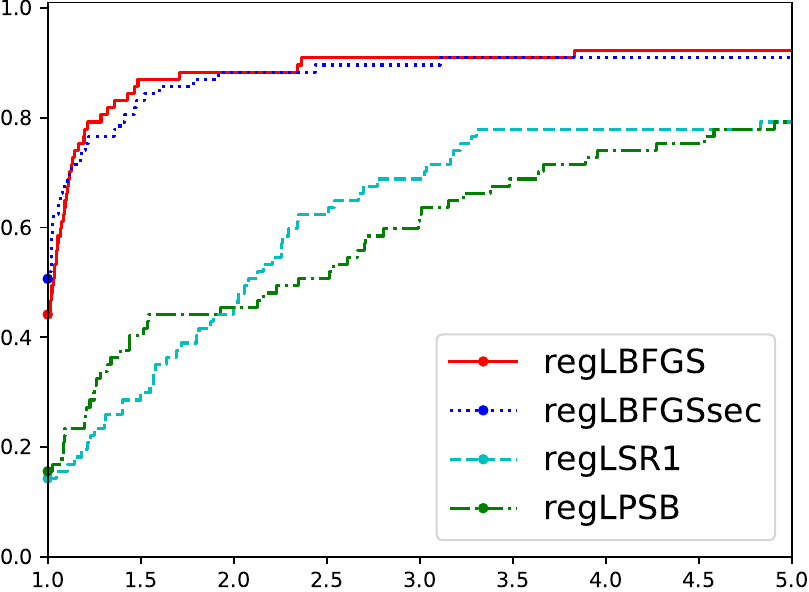}
    ~~
    \includegraphics[width=0.45\textwidth]{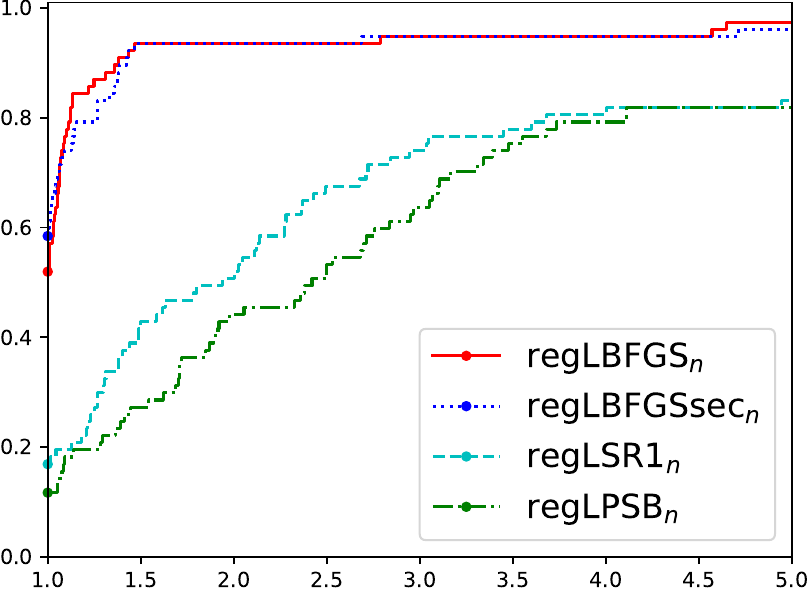}
    \caption{Performance profiles based on the number of function evaluations for the four algorithms from Section~\ref{Sec:Numerics:Reg}: monotone case (left), nonmonotone case (right).}
    \label{Fig:PerfProf11}
\end{figure}

The algorithms were terminated as soon as either
\begin{equation*}
    \|\mvec{g}_k\|_{\infty}<10^{-4}, \qquad
    k \ge 10^5, \qquad\text{or}\quad
    \RegParam_k > 10^{15}.
\end{equation*}
The initial estimate $\InitialB$ in step $k$ is defined by the standard formula
\begin{equation}\label{Eq:InitialB}
    \InitialB = \gamma_k \mmat{I}, \qquad
    \gamma_k = \frac{\mvec{y}_k^{\Trp}\mvec{y}_k}{\mvec{y}_k^{\Trp}\mvec{s}_k}.
\end{equation}
In addition, we adopted a lower threshold $\RegParam_{\min}:=10^{-4}$ for the regularization parameter. This improved the practical behavior of the method (particularly in the L-BFGS case) and also prevented the regularization parameter from becoming zero in limited-precision arithmetic.

It may seem that the above choices lead to a preference of high regularization parameters over low ones and could therefore impede fast asymptotic convergence. What we have found empirically is that Algorithm~\ref{Alg:RegularizedQN} (with L-BFGS) often behaves best when the regularization parameter is changed infrequently. This suggests that the parameter should be increased sharply when necessary (to avoid having to increase repeatedly), and only decreased when the step quality is very good. This is reflected in our choice of parameters.

Note also that limited memory methods rarely achieve actual superlinear convergence; the typical behavior is asymptotically linear \cite{Liu1989}, and classical results for inexact Newton methods (e.g., \cite[Thm.~7.1]{Nocedal2006}) indicate that a small but non-decaying value of $\RegParam_k$ will typically preserve linear convergence. This indicates that the choices made here are sound from a theoretical point of view.

Comparable studies in other papers \cite{Sugimoto2014,Burdakov2017} indicate that regularized methods may benefit from a nonmonotonicity strategy. Therefore, and to obtain a larger dataset, we also implemented nonmonotone versions of all algorithms, where $M:=8$ was chosen as the nonmonotonicity offset; this was incorporated into the methods by replacing the reference value $f(\mvec{x}_k)$ in the regularization control \eqref{Eq:ARed} and the line search routines by $\max_{0\le i < M} f(\mvec{x}_{k-i})$ for $k\ge M$. The initial steps $k=0,\ldots,M-1$ were treated without modification.

\begin{figure}\centering
    \includegraphics[width=0.45\textwidth]{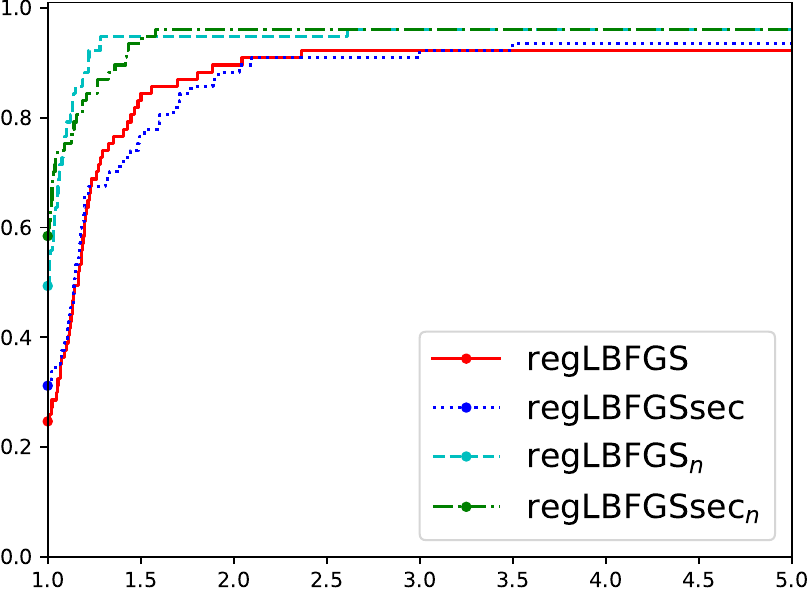}
    ~~
    \includegraphics[width=0.45\textwidth]{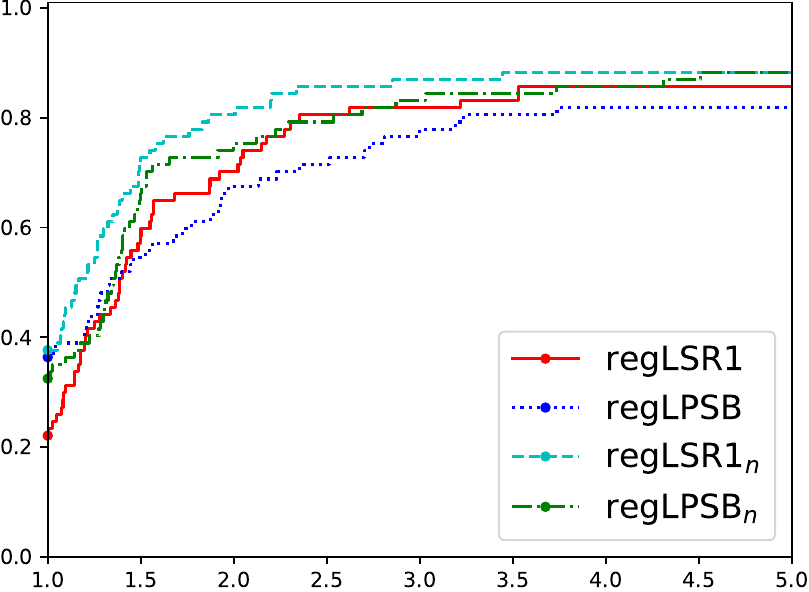}
    \caption{Performance profiles based on the number of function evaluations for the four algorithms from Section~\ref{Sec:Numerics:Reg}: monotone vs.\ nonmonotone (index $n$) algorithms.}
    \label{Fig:PerfProf12}
\end{figure}

Figure~\ref{Fig:PerfProf12} illustrates the relative behavior of the monotone and nonmonotone implementations. All algorithms seem to benefit from a nonmonotonicity strategy. It should be emphasized that our choice of such strategy is rather simple but we believe it is sufficient to illustrate the general picture.

Overall, somewhat unsurprisingly, L-BFGS turns out to be by far the most efficient quasi-Newton scheme even in the context of regularization. The regularized variants of L-SR1 and L-PSB are moderately competitive but fall short of the overall performance of \textsf{regLBFGS} and \textsf{regLBFGSsec}.

An interesting observation we made during our testing is that L-SR1 and, in particular, L-PSB were actually more efficient when used with a more ``optimistic'' regularization scheme (i.e., lower regularization parameters). This is somewhat surprising because these methods generate indefinite Hessian approximations which should, intuitively, benefit the most from regularization; on the other hand, L-BFGS generates an approximation which is positive definite anyway, which suggests that regularization may be less necessary here. The numerical evidence we observed contradicts this intuition.

We can only give a partial explanation for this phenomenon. It is well-known that BFGS and L-BFGS are related to the classical conjugate gradient method \cite{Nazareth1979}, which suggests that L-BFGS imposes some kind of relationship (a generalized ``conjugacy'') on successive search directions (see also the discussion after \cite[Eq.~65]{Burdakov2017}). We are unaware of a rigorous definition of such a property, but the relationship of successive search directions may be preserved in a certain way when L-BFGS is used with a regularization parameter that changes infrequently. On the other hand, L-SR1 and L-PSB are generally considered to generate more accurate approximations of the exact Hessian (especially when it is indefinite), which indicates that these methods behave more similarly to a conventional Newtonian algorithm and therefore benefit from a quicker reduction of regularization parameters.

The regularized secant version \textsf{regLBFGSsec} is interesting because it is rather simple to implement (by using the standard two-loop recursion) yet rivals the robustness of \textsf{regLBFGS}; see Figure~\ref{Fig:PerfProf11}.

\begin{table}\centering
\begin{tabular}{lllll}
    \toprule
    & \textsf{regLBFGS} & \textsf{regLBFGSsec} & \textsf{regLSR1} & \textsf{regLPSB} \\
    \midrule
    \textnormal{\% accepted steps (monotone)} & 84\% & 82\% & 85\% & 99\% \\
    \textnormal{\% accepted steps (nonmonotone)} & 99\% & 99\% & 87\% & 72\% \\
    \midrule
    \textnormal{\# problems solved (monotone)} & 72 & 72 & 68 & 67 \\
    \textnormal{\# problems solved (nonmonotone)} & 75 & 75 & 73 & 74 \\
    \bottomrule
\end{tabular}
\caption{Average proportion of accepted steps and total problems solved for all algorithms from Section~\ref{Sec:Numerics:Reg}.}
\label{Tab:SuccessfulSteps1}
\end{table}

Finally, Table~\ref{Tab:SuccessfulSteps1} shows the proportion of accepted steps and the total number of solved problems for all four regularization-based algorithms. The L-BFGS algorithms stand out for their high number of solved problems overall, and acceptance ratios of around 99\% in the nonmonotone case. Interestingly, \textsf{regLPSB} achieves around 99\% acceptance rate in the monotone case, tapering off to around 72\% for the nonmonotone implementation.

\subsection{Comparison to Existing Algorithms}\label{Sec:Numerics:Existing}

Let us now measure \textsf{regLBFGS} against relevant algorithms available in the literature. The ``reference'' algorithms we use are:
\begin{itemize}[leftmargin=8em]
    \item[{\bfseries\sffamily armijoLBFGS:}] the ordinary L-BFGS method with Armijo line search and the cautious updating scheme \eqref{Eq:CautiousUpdateLBFGS};
    \item[{\bfseries\sffamily wolfeLBFGS:}] the Liu--Nocedal L-BFGS method \cite{Liu1989} with Mor\'e--Thuente line search \cite{More1994};
    \item[{\bfseries\sffamily eigLBFGS:}] a slightly simplified version of the EIG$(\infty, 2)$ trust region L-BFGS algorithm from \cite{Burdakov2017}.
\end{itemize}
The Armijo search uses standard backtracking by repeatedly halving the step size $t_k$ until
\begin{equation*}
f(\mvec{x}_k + t_k \mvec{d}_k) \le f(\mvec{x}_k) + c_1 t_k \mvec{g}_k^{\Trp}\mvec{d}_k,
\end{equation*}
where (in this context) $\mvec{d}_k$ is the quasi-Newton step. The Mor\'e--Thuente line search uses the implementation of Diane O'Leary \cite{OLeary1991}, translated into \texttt{Python}. It terminates when
\begin{equation*}
    f(\mvec{x}_k + t_k \mvec{d}_k) \le f(\mvec{x}_k) + c_1 t_k \mvec{g}_k^{\Trp}\mvec{d}_k
    \quad\textnormal{and}\quad
    | \nabla f(\mvec{x}_k + t_k \mvec{d}_k)^{\Trp} \mvec{d_k} | \le -0.9 \mvec{g}_k^{\Trp}\mvec{d}_k.
\end{equation*}
The \textsf{eigLBFGS} algorithm is based on the EIG$(\infty, 2)$ implementation available at \url{https://gratton.perso.enseeiht.fr/LBFGS/index.html}.
In order to make the comparison fair, we have slightly simplified this algorithm by replacing the two-stage initial line search from EIG$(\infty, 2)$ with a single Mor\'e--Thuente search along the normalized negative gradient direction (which is consistent with the implementation of the regularization methods; see Section~\ref{Sec:Numerics:Reg}). Furthermore, the stopping criteria, cautious update mechanism, and trust region control parameters of EIG$(\infty, 2)$ were brought in line with the other implementations.

\begin{figure}\centering
    \includegraphics[width=0.45\textwidth]{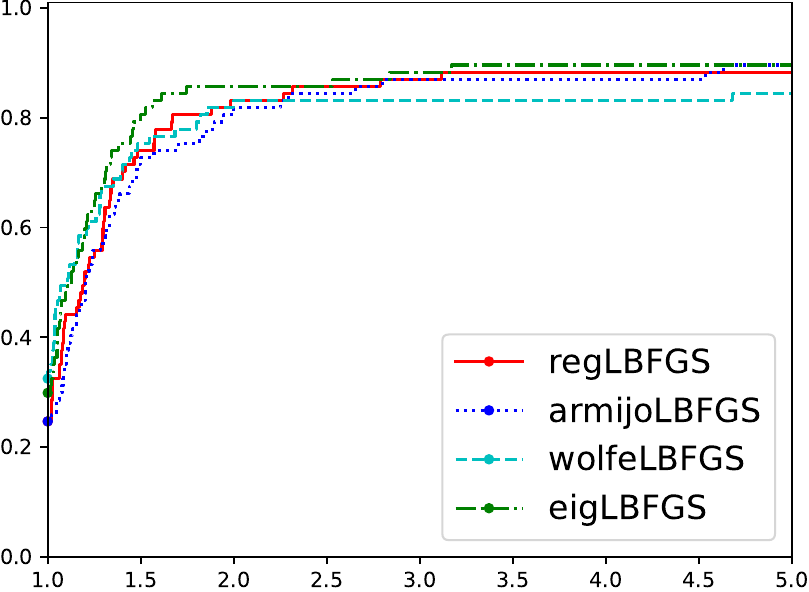}
    ~~
    \includegraphics[width=0.45\textwidth]{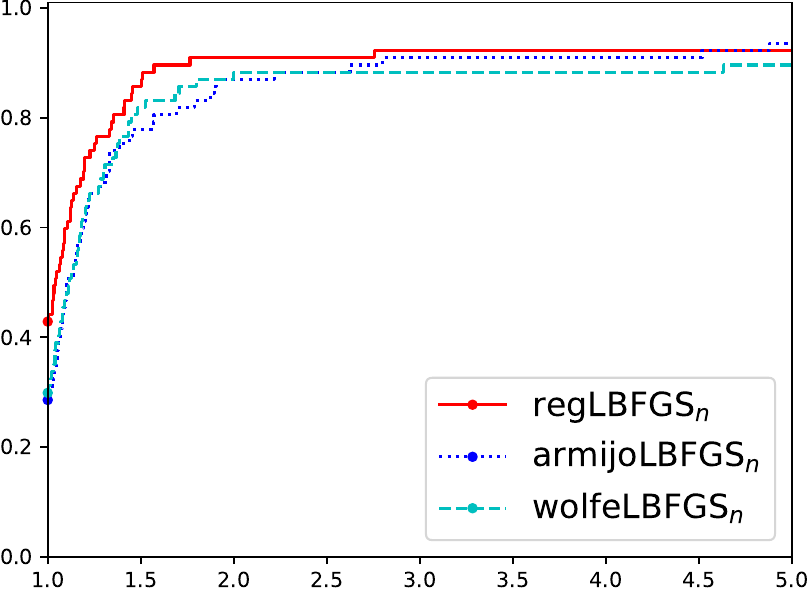}
    \caption{Performance profiles based on the number of function evaluations for \textsf{regLBFGS} and the three algorithms from Section~\ref{Sec:Numerics:Existing}: monotone case (left), nonmonotone case (right).}
    \label{Fig:PerfProf21}
\end{figure}

Note that a nonmonotone implementation of the EIG$(\infty, 2)$ algorithm from \cite{Burdakov2017} is not available, so we have excluded it from the corresponding comparisons.

The algorithms in this section all use the stopping criteria
\begin{equation*}
\|\mvec{g}_k\|_{\infty}<10^{-4}, \qquad
k \ge 10^5, \qquad\text{or}\quad
\begin{cases}
\Delta_k < 10^{-15}, \\
t_k < 10^{-15},
\end{cases}
\end{equation*}
depending on whether the algorithm is of line search or trust region type. Here, $t_k$ is the line search step size and $\Delta_k$ denotes the trust-region radius.

Figure~\ref{Fig:PerfProf21} illustrates the performance of the three algorithms mentioned above and \textsf{regLBFGS}. The L-BFGS algorithm with Mor\'e--Thuente line search \cite{More1994} is competitive on the fastest problems. Similar to the results in \cite{Burdakov2017}, however, we found that this and similar Wolfe--Powell based algorithms were noticeably less efficient than others due to the excessive number of function evaluations.

\textsf{regLBFGS} and \textsf{eigLBFGS} perform very similarly in the monotone case, with \textsf{eigLBFGS} (the algorithm based on \cite{Burdakov2017}) attaining a slight advantage. This is not entirely surprising as regularization can be seen as an approximation of trust region algorithms. In return, \textsf{eigLBFGS} requires a (low-dimensional) eigenvalue decomposition in every iteration where the trust region is active, whereas \textsf{regLBFGS} only solves a symmetric linear equation.

\begin{figure}\centering
    \includegraphics[width=0.45\textwidth]{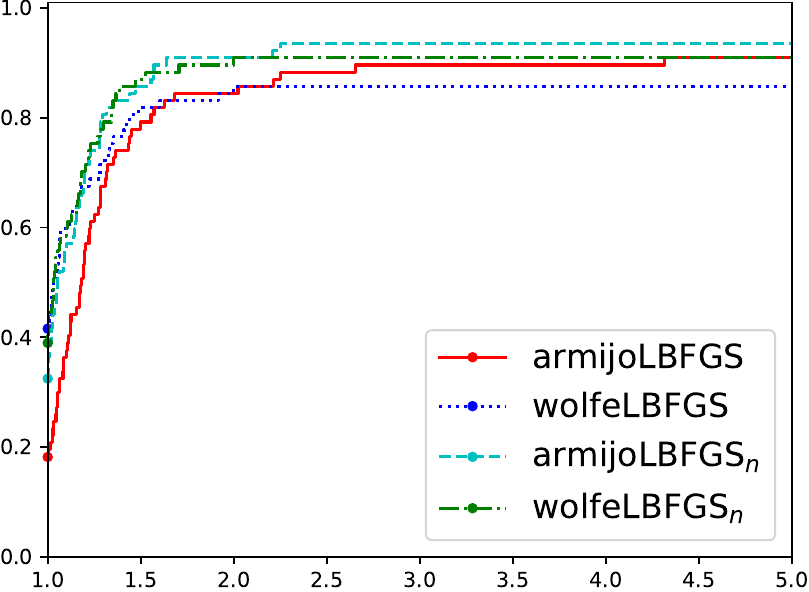}
    ~~
    \includegraphics[width=0.45\textwidth]{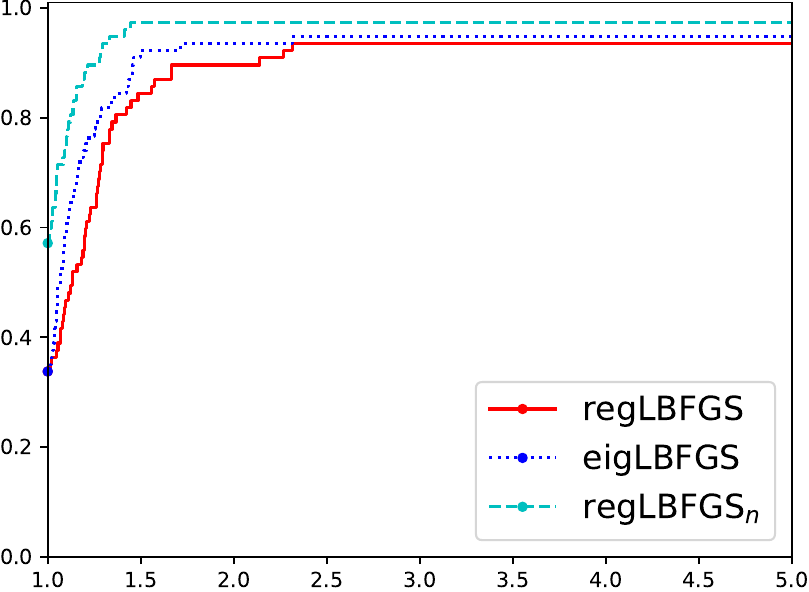}
    \caption{Performance profiles based on the number of function evaluations for \textsf{regLBFGS} and the three algorithms from Section~\ref{Sec:Numerics:Existing}: monotone vs.\ nonmonotone (index $n$) algorithms.}
    \label{Fig:PerfProf22}
\end{figure}

Figure~\ref{Fig:PerfProf22} compares the behavior of monotone and nonmonotone algorithms. The nonmonotone version of \textsf{regLBFGS} seems to outperform both \textsf{eigLBFGS} and nonmonotone versions of \textsf{armijoLBFGS} and \textsf{wolfeLBFGS} (see Figure~\ref{Fig:PerfProf21}).

Note again that our comparison above is based exclusively on function evaluations, not CPU times. It may be interesting to also conduct an analysis of CPU times, but this would effectively require another programming language due to the lack of optimizing compilation in languages like \texttt{Python} or \texttt{MATLAB}, which incurs significant overhead on loops and repeated assignment operations. We anticipate that realistic CPU times would slightly benefit the line search L-BFGS methods due to the logistic effort associated with limited memory updating in the regularized methods (see Section~\ref{Sec:Matrices:BFGS}).


\begin{table}\centering
\begin{tabular}{lllll}
    \toprule
    & \textsf{regLBFGS} & \textsf{armijoLBFGS} & \textsf{wolfeLBFGS} & \textsf{eigLBFGS} \\
    \midrule
    \textnormal{\% accepted steps (monotone)} & 84\% & 59\% & 95\% & 96\% \\
    \textnormal{\% accepted steps (nonmonotone)} & 99\% & 82\% & 95\% & --- \\
    \midrule
    \textnormal{\# problems solved (monotone)} & 72 & 70 & 67 & 73 \\
    \textnormal{\# problems solved (nonmonotone)} & 75 & 73 & 71 & --- \\
    \bottomrule
\end{tabular}
\caption{Average proportion of accepted steps and total problems solved for all algorithms from Section~\ref{Sec:Numerics:Existing}.}
\label{Tab:SuccessfulSteps2}
\end{table}

Finally, Table~\ref{Tab:SuccessfulSteps2} shows the average ratio of accepted steps and total number of solved problems for all algorithms from this section, and \textsf{regLBFGS}. The interpolation-based Mor\'e--Thuente line search achieves around 95\% acceptance in the monotone and nonmonotone implementations. Somewhat unsurprisingly, the trust-region based \textsf{eigLBFGS} algorithms achieves the highest acceptance rate in the monotone case. \textsf{regLBFGS} again stands out with 99\% acceptance in the nonmonotone case.

\begin{remark}[Further improvements]\label{Rem:FurtherImprovements}
    It is possible to incorporate further modifications and improvements into the regularized quasi-Newton schemes, but we have abstained from doing so in order to facilitate a fair comparison. For instance, it may be beneficial to update the quasi-Newton information in \emph{rejected} steps since the trial function value and gradient provide meaningful information \cite{Sugimoto2014}. Note that this technique is covered by the framework of Algorithm~\ref{Alg:RegularizedQN} since we allow $\mmat{B}_k$ to be chosen anew in each iteration.
\end{remark}

\section{Final Remarks}\label{Sec:Final}

The results and numerical evidence in this paper demonstrate conclusively that regularization is a powerful globalization technique for limited memory quasi-Newton methods.

The numerical results in particular indicate that regularization techniques can substantially improve the efficiency and robustness of L-BFGS on large-scale nonlinear problems or when nonmonotonicity strategies are employed. An intuitive explanation of this phenomenon lies in the fact that regularization ``stabilizes'' the Hessian approximation in the sense that the condition number becomes smaller, which may make the method less susceptible to step jumps or ``discontinuities'' induced by nonmonotonicity or extreme nonlinearity.

We hope that the findings presented here will facilitate more research into these techniques, for example, on quantitative convergence results or on how to integrate regularization with BFGS in a full-memory context.

\phantomsection
\addcontentsline{toc}{section}{\refname}

{\footnotesize
\bibliographystyle{abbrv}
\bibliography{LargeScaleQN}
}

\end{document}